\documentclass{article}

\usepackage[utf8]{inputenc}
\usepackage{bm,amsfonts,amsthm,amsmath,amssymb}
\usepackage{graphicx,color}
\usepackage{mathrsfs}
\usepackage{indentfirst}
\usepackage{bbm}
\usepackage{appendix}

\DeclareMathAlphabet{\mathpzc}{OT1}{pzc}{m}{it}

\newtheorem{theorem}{Theorem}

\newtheorem{lemma}[theorem]{Lemma}

\numberwithin{theorem}{section}
\numberwithin{equation}{section}

\allowdisplaybreaks[4]
\title{The nonconforming Crouzeix-Raviart element approximation and two-grid discretizations for the elastic eigenvalue problem}
\author{Hai Bi\footnote{ School of Mathematical Science, Guizhou Normal University, Guiyang 550025, China (bihaimath@gznu.edu.cn) }, Xuqing Zhang \footnote{School of Mathematical Science, Guizhou Normal University, Guiyang 550025, China (zhxuqing1230@126.com)}, Yidu Yang
\footnote{School of Mathematical Science, Guizhou Normal University, Guiyang 550025, China (ydyang@gznu.edu.cn)}}

\date{}

\begin{document}
\maketitle
\begin{abstract}
  In this paper, we extend the work of Brenner and Sung [Math. Comp. 59, 321--338 (1992)] and present a regularity estimate for the elastic equations in concave domains. Based on the regularity estimate we prove that the constants in the error estimates of the nonconforming Crouzeix-Raviart element approximations for the elastic equations/eigenvalue problem are independent of Lam$\acute{e}$ constant, which means the nonconforming Crouzeix-Raviart element approximations are locking-free. We also establish two kinds of two-grid discretization schemes for the elastic eigenvalue problem, and analyze that when the mesh sizes of coarse grid and fine grid satisfy some relationship, the resulting solutions can achieve the optimal accuracy. Numerical examples are provided to show the efficiency of two-grid schemes for the elastic eigenvalue problem.
\end{abstract}

%\tableofcontents

\section{Introduction}

Due to the wide application background, the approximate computation for elastic equations/eigenvalue problems
 has attracted the attention of academic circles, for instance, \cite{falk1991,Babuska1992,brenner1992,Kouhia,Zhang,Ovtchinnikov,Ming,Wang2002,oden,LLS,Wang2004,Walsh,Hernandez,Meddahi,Russo,Lee,gong,brenner}, etc. It is known that for numerical solutions of the equations of linear isotropic planar elasticity, standard conforming
finite elements suffer a deterioration in performance as the Lam$\acute{e}$ constant $\lambda\rightarrow \infty$,
that is locking phenomenon (see \cite{Babuska1992,BS1992}). To overcome  the locking phenomenon, several numerical approaches have been developed. For example, the $p$-version method \cite{Vogelius}, the PEERS method \cite{ABD}, the mixed method \cite{Stenberg}, the Galerkin least squares method  \cite{FS}, the nonconforming triangular elements \cite{falk1991,brenner1992} and quadrilateral elements \cite{Zhang,LLS,Wang2004,Ming}, and so on.

For the computation of eigenvalue problems in elasticity, there have been quite a few studies. For instance, \cite{Ovtchinnikov} adopts a preconditioning technique associated with dimensional reduction algorithm for the thin elastic structures.
\cite{oden} presents a method for three-dimensional linear elasticity or shell problems to
derive computable estimates of the approximation error in eigenvalues.
\cite{Walsh} develops an a posteriori error estimator for linearized elasticity eigenvalue problems.
\cite{Hernandez} analyzes the finite element approximation of the spectral problem for the
linear elasticity equation with mixed boundary conditions in a curved concave domain.
\cite{Meddahi} conducts an analysis for the eigenvalue problem of linear elasticity by means of a mixed variational formulation.
\cite{Russo} presents a theory for the approximation of eigenvalue problems in mixed form by nonconforming methods and apply
it to the classical Hellinger-Reissner mixed formulation for a linear elastic structure, etc.
Recently, \cite{Lee} uses the immersed finite element method based on Crouzeix-Raviart (C-R)
P1-nonconforming element to approximate eigenvalue problems for elasticity equations with interfaces.
\cite{gong} explores a shifted-inverse adaptive multigrid method for the elastic eigenvalue problem.

In the above literatures, \cite{falk1991,brenner1992,Lee,brenner} study the nonconforming C-R
element method for the elastic equations/eigenvalue problems in convex domains, and as far as we know, there is no report on the nonconforming C-R approximation for the elastic eigenvalue problems in concave domain.
In this paper, we extend the work in \cite{brenner1992,brenner} and present a regularity estimate for the elastic equations in concave domain (see (\ref{s2.8})). Since in the standard error analysis for the consistency term, it is required that the ``minimum" regularity $\mathbf{u}\in \mathbf{H}^{1+s}(\Omega)$ for $s\geq\frac{1}{2}$ which is not necessarily satisfied in concave domain, \cite{gudi2010,mao2010} adopt a new method to conduct the error estimate for the C-R element approximation. To be more specific, they made use of the conforming interpolation of the nonconforming C-R element approximation. However, at present we cannot use their method to warrant the error estimates are locking-free for the elastic eigenvalue problem. So, we adopt the argument in \cite{bernardi,caiz} to prove a trace inequality in which the constant is analyzed elaborately (see Lemma 3.3) with the condition slightly different from that in the existing literatures and then derive the estimates of consistency term. Based on the regularity estimate we prove that the constants in the error estimates of the nonconforming C-R element approximations for the elastic equations/eigenvalue problem are independent of the Lam$\acute{e}$ constant, which means the C-R element approximations are locking-free.

Since introduced by Xu and Zhou \cite{xu1992,xu1996}, due to the good performance in reducing computational costs and improving accuracy, the two-grid discretization method has been developed and successfully applied to other problems, for instance, Poisson equation/integral equation eigenvalue problems \cite{xu2001,yang2011}, semilinear eigenvalue problem \cite{cj}, Stokes equations \cite{li,cai,mu}, Schr$\ddot{o}$dinger equation \cite{clz,gsz}, quantum eigenvalue problem \cite{dai}, Steklov eigenvalue problem \cite{xie2014,by} and so on.
In this paper, we establish two kinds of two-grid discretization schemes of nonconforming C-R element. We prove that the constants in error estimates are independent of the Lam$\acute{e}$ constants, i.e., the two-grid discretization schemes of nonconforming C-R element are also locking-free, and when the mesh sizes of coarse grid and fine grid satisfy some relationship, the resulting solutions can achieve the optimal accuracy. We present some numerical examples to show the two-grid discretization schemes are efficient for solving elastic eigenvalue problem.

 The rest of the paper is organized as follows. Some preliminaries are given in Section 2. The nonconforming C-R element approximation for the elastic eigenvalue problem is established in Section 3. Two-grid discretization schemes and the corresponding error analysis are presented in Section 4. Finally, numerical experiments are shown in Section 5.

 We refer to \cite{Babuska,boffi,brenner,ciarlet} as regards the basic theory of finite element methods in this paper.

Throughout this paper, we use the letter $C$, with or without subscripts, to denote a generic positive constant independent of the Lam$\acute{e}$ constants $\mu$, $\lambda$ and the mesh size $h$, % with its dependencies listed as its subscripts.
which may take different values in different contexts.

\section{Preliminaries}\label{Preliminary}

Let $\mathbf{x}=(x,y)^T\in \mathbb{R}^2, \Omega\subset \mathbb{R}^2$ be a bounded Lipschitz polygon but not necessarily convex. The standard notation $W^{s,p}(\Omega)$ is used to denote Sobolev spaces, and $H^s(\Omega)$ and their associated norms $\|\cdot\|_{s,\Omega}$ and seminorms $|\cdot|_{s,\Omega}$ are used in the case of $p=2$. Denote $H^1_0(\Omega)=\{v\in H^1(\Omega):v|_{\partial\Omega}=0\}$ where $v|_{\partial\Omega}=0$ is in the sense of trace. The space $H^{-s}(\Omega)$, the dual of $H^s(\Omega)$, will also be used.
In this paper, the bold letter is used for vector-valued functions and their associated spaces, and the following conventions are adopted for the Sobolev norms and seminorms: for any $\mathbf{v}=(v_1(\mathbf{x}), v_2(\mathbf{x}))^T\in \mathbf{H}^s(\Omega)$,
\begin{equation*}
\begin{aligned}
&\|\mathbf{v}\|_{\mathbf{H}^s(\Omega)}:=(\|v_1\|_{s,\Omega}^2+\|v_2\|_{s,\Omega}^2)^{\frac{1}{2}},\\
&|\mathbf{v}|_{\mathbf{H}^s(\Omega)}:=(|v_1|_{s,\Omega}^2+|v_2|_{s,\Omega}^2)^{\frac{1}{2}}.
\end{aligned}
\end{equation*}
Bold letter with an undertilde is used for matrix-valued functions and spaces. For matrix-valued function $A=(a_{ij})_{1\leq i,j\leq 2}$,
\begin{equation*}
\|A\|_{\underset{\sim}{\mathbf{H}}^s(\Omega)}:=(\sum_{i,j=1}^{2}\|a_{ij}\|_{s,\Omega}^2)^{\frac{1}{2}}.
\end{equation*}
\indent The elastic eigenvalue problem is to find $\omega\in \mathbb{R}$ and $\mathbf{u}\neq 0$ such that
\begin{equation}\label{s2.1}
\left\{
\begin{array}{ll}
-\nabla \cdot \sigma(\mathbf{u})=\omega\rho \mathbf{u}~~in~ \Omega,\\
\mathbf{u}=0~~~on~ \partial \Omega.
\end{array}
\right.
\end{equation}
Here $\mathbf{u}(\mathbf{x})=(u_1(\mathbf{x}), u_2(\mathbf{x}))^T$ is the displacement vector, $\rho(\mathbf{x})$ is the mass density, and $\sigma(\mathbf{u})$ is the stress tensor given by the generalized Hooke law
\begin{equation*}
\sigma(\mathbf{u})=2\mu\varepsilon(\mathbf{u})+\lambda tr(\varepsilon(\mathbf{u}))I,
\end{equation*}
where $I\in \mathbb{R}^{2\times 2}$ is the identity matrix, and the positive constants $\mu, \lambda$ denote the Lam$\acute{e}$ parameters satisfying $(\mu, \lambda)\in [\mu_0, \mu_1]\times (0, +\infty)$ where $0<\mu_0<\mu_1<\infty$.
The strain tensor $\varepsilon(\mathbf{u})$ is defined as
\begin{equation*}
\varepsilon(\mathbf{u})=\frac{1}{2}(\nabla \mathbf{u}+(\nabla \mathbf{u})^T),
\end{equation*}
where $\nabla \mathbf{u}$ is the displacement gradient tensor
\begin{equation*}
\nabla \mathbf{u}=\left [
  \begin{array}{cc}
    \partial_{x}u_{1} & \partial_{y}u_{1} \\
    \partial_{x}u_{2} & \partial_{y}u_{2} \\
  \end{array}
\right ].
\end{equation*}

 The weak form for \eqref{s2.1} is stated as to find $(\omega, \mathbf{u})\in \mathbb{R}\times \mathbf{H}^1_0(\Omega)$,  $\|\mathbf{u}\|_{\mathbf{H}^1(\Omega)}=1$,
such that
\begin{equation}\label{s2.2}
a(\mathbf{u},\mathbf{v})=\omega b(\mathbf{u},\mathbf{v}),~~~\forall \mathbf{v}\in \mathbf{H}_{0}^{1}(\Omega),
\end{equation}
where
\begin{align}
a(\mathbf{u},\mathbf{v})&=\int_{\Omega }\sigma(\mathbf{u}): \nabla \mathbf{v}d\mathbf{x} \nonumber\\
& =\int _{\Omega }(\mu\mathbf{\nabla}\mathbf{u}:\mathbf{\nabla}\mathbf{v}+(\mu+\lambda)(div\mathbf{u})(div\mathbf{v}))d\mathbf{x}\label{s2.3}\\
& =\int _{\Omega}(2\mu\varepsilon(\mathbf{u}):\varepsilon(\mathbf{v})+\lambda div \mathbf{u}~div\mathbf{v})d\mathbf{x},\nonumber\\
b(\mathbf{u},\mathbf{v})&=\int _{\Omega}\rho\mathbf{u}\cdot\mathbf{v} d\mathbf{x}=\int _{\Omega}\rho\sum_{i=1}^{2}u_{i}v_{i}d\mathbf{x}.\nonumber%\label{s2.4}
\end{align}
Here $A:B=tr(AB^T)$ is the Frobenius inner product of matrices $A$ and $B$.
It can be verified that the above bilinear form $a(\cdot,\cdot)$ and the linear form $b(\cdot,\cdot)$ are continuous over the
space $\mathbf{H}_{0}^{1}(\Omega)$ and $\mathbf{L}^{2}(\Omega)$, respectively, and from Korn's inequality it can be proved that $a(\cdot,\cdot)$ is $\mathbf{H}_{0}^{1}$-elliptic. Thus, $a(\cdot,\cdot)$ and $\|\cdot\|_a=\sqrt{a(\cdot,\cdot)}$ can be used as an inner product and norm on $\mathbf{H}_{0}^{1}(\Omega)$. Without loss of generality, we assume that $\rho\equiv 1$ in the rest of the paper.

The source problem associated with  \eqref{s2.2} is:
Find $\mathbf{w}\in \mathbf{H}^1_0(\Omega)$ such that
\begin{equation}\label{s2.4}
a(\mathbf{w},\mathbf{v})= b(\mathbf{f},\mathbf{v}),~~~\forall \mathbf{v}\in \mathbf{H}_{0}^{1}(\Omega).
\end{equation}
\indent In \cite{brenner1992,brenner} Brenner et al. study and prove the following the a priori estimates for \eqref{s2.4} when $\Omega$ is convex:
\begin{equation}\label{s2.5}
\|\mathbf{w}\|_{\mathbf{H}^{2}(\Omega)}+\lambda \|div\mathbf{w}\|_{1,\Omega}\leq C_{\Omega}\|\mathbf{f}\|_{\mathbf{L}^{2}(\Omega)}.
\end{equation}
Next, using the argument in \cite{brenner1992} we shall discuss the a priori estimates for \eqref{s2.4} when $\Omega$ is concave. In this case,
it needs more delicate analysis since the solution of  \eqref{s2.4} is not smooth enough.
\begin{lemma}\label{lem2.1}
For any given $\mathbf{w}\in\mathbf{H}^{1+t}(\Omega)\cap\mathbf{H}_{0}^{1}(\Omega)~(0\leq t \leq 1)$,
there exists $\mathbf{w}^{*}\in\mathbf{H}^{1+t}(\Omega)\cap\mathbf{H}_{0}^{1}(\Omega)$ such that
\begin{align}
&div\mathbf{w}^{*}=div\mathbf{w},\label{s2.6}\\
&\|\mathbf{w}^{*}\|_{\mathbf{H}^{1+t}(\Omega)}\leq C\|div \mathbf{w}\|_{t,\Omega}.\label{s2.7}
\end{align}
\end{lemma}
\begin{proof}
Since $\mathbf{w}\in\mathbf{H}^{1+t}(\Omega)\cap\mathbf{H}_{0}^{1}(\Omega)$,
$div \mathbf{w}\in H^{t}(\Omega)$ and $\int_{\Omega}div \mathbf{w}d\mathbf{x}=0$,
by Theorem 3.1 in \cite{arnold} we know that there exists $\mathbf{w}^*\in\mathbf{H}^{1+t}(\Omega)\cap\mathbf{H}_{0}^{1}(\Omega)$
such that \eqref{s2.6} and \eqref{s2.7} hold.
\end{proof}

\begin{theorem}\label{thm2.1}
For $\mathbf{f}\in\mathbf{L}^{2}(\Omega)$,
\eqref{s2.4} has a unique solution $\mathbf{w}\in\mathbf{H}^{1+s}(\Omega)$ and $\mathbf{w}\in \mathbf{W}^{2,p}(\Omega)~(p=\frac{2}{2-s})$, and there exists a positive constant $C_{\Omega}$ such that
\begin{equation}\label{s2.8}
\|\mathbf{w}\|_{\mathbf{H}^{1+s}(\Omega)}+\lambda \|div\mathbf{w}\|_{s,\Omega}\leq C_{\Omega}\|\mathbf{f}\|_{\mathbf{L}^{2}(\Omega)},
\end{equation}
where $s<\frac{1}{2}$ and $s$ can be close to $\frac{1}{2}$ arbitrarily, and $C_{\Omega}$ is the a priori constant dependent on $\Omega$ but independent of $\mu, \lambda$ and $\mathbf{f}$.
\end{theorem}
\begin{proof}
Since $a(\cdot,\cdot)$ is $\mathbf{H}_0^{1}$-elliptic and $b(\cdot,\cdot)$ is continuous, from the Lax-Milgram theorem we know that
\eqref{s2.4} admits a unique solution $\mathbf{w}\in \mathbf{H}_0^{1}(\Omega)$.

From Theorem 4.2.5 in \cite{grisvard2} and \cite{farhloul} we know that there exist numbers $C_{i,z}$ such that
\begin{equation}\label{s2.9}
\mathbf{w}=\mathbf{w}_{0}+ \sum\limits_{i=1}^{N}\sum\limits_{z}C_{i,z}r_{i}^{z}\phi_{i,z}(\vartheta_{i}),
\end{equation}
where $\mathbf{w}_{0}\in \mathbf{H}^{2}(\Omega)$, $N$ is the number of corners of $\Omega$, $r_{i}$ is the distance from any point to the $i$th corner of $\Omega$, $z\in (0,1)$ is a real solution of
\begin{equation*}
sin^{2}(z\vartheta_{i})=z^{2}sin^{2}\vartheta_{i}
\end{equation*}
and $min\{z\}>\frac{1}{2}$, and $\phi_{i,z}(\vartheta_{i})$ is a vector field depending on $z, \lambda, \mu$ and sine and cosine function of interior angle $\vartheta_{i}$ at the $i$th corner (the expression of $\phi_{i,z}(\vartheta_{i})$ we refer to (4.2.14) in \cite{grisvard2}).

 From \eqref{s2.9} we can see that the singularity of $\mathbf{w}$ depends on $r_{i}^{z}$, thus we know that $\mathbf{w}\in \mathbf{H}^{1+t}(\Omega)$ for all $t\in (\frac{1}{2},min\{z\})$, and $\mathbf{w}\in \mathbf{W}^{2,p}(\Omega)$ with $p=\frac{2}{2-t}$.

 Next, we shall prove \eqref{s2.8}.
Let $\mathbf{v}=\mathbf{w}$ in \eqref{s2.4}, from  \eqref{s2.3} we have
\begin{equation}\label{s2.10}
2\mu\int _{\Omega }\varepsilon(\mathbf{w}):\varepsilon(\mathbf{w})d\mathbf{x}\leq \|\mathbf{f}\|_{\mathbf{L}^{2}(\Omega)}\|\mathbf{w}\|_{\mathbf{L}^{2}(\Omega)}.
\end{equation}
By using First Korn inequality (cf. Corollary 11.2.25 in \cite{brenner}) and  \eqref{s2.10} we deduce
\begin{align}
\|\mathbf{w}\|^2_{\mathbf{H}^{1}(\Omega)}&\leq C_{\Omega}\|\varepsilon(\mathbf{w})\|^2_{\underset{\sim}{\mathbf{L}}^2(\Omega)}\leq C_{\Omega}\|\mathbf{f}\|_{\mathbf{L}^{2}(\Omega)}\|\mathbf{w}\|_{\mathbf{L}^{2}(\Omega)} \nonumber\\
&\leq C_{\Omega}\|\mathbf{f}\|_{\mathbf{L}^{2}(\Omega)}\|\mathbf{w}\|_{\mathbf{H}^{1}(\Omega)},\nonumber
\end{align}
i.e.,
\begin{equation}\label{s2.11}
\|\mathbf{w}\|_{\mathbf{H}^{1}(\Omega)}\leq C_{\Omega}\|\mathbf{f}\|_{\mathbf{L}^{2}(\Omega)}.
\end{equation}
From Lemma 2.1 we know that there exists $\mathbf{w}^{*}\in\mathbf{H}_{0}^{1}(\Omega)$ such that
\begin{align}
&div\mathbf{w}^{*}=div\mathbf{w},\label{s2.12}\\
&\|\mathbf{w}^{*}\|_{\mathbf{H}^{1}(\Omega)}\leq C_{\Omega}\|div\mathbf{w}\|_{0,\Omega}.\label{s2.13}
\end{align}
Taking  $\mathbf{v}=\mathbf{w}^*$ in \eqref{s2.4} and using \eqref{s2.12} we deduce
\begin{equation*}
\lambda\int_{\Omega}|div\mathbf{w}|^{2}d\mathbf{x}\leq \|\mathbf{f}\|_{\mathbf{L}^{2}(\Omega)}\|\mathbf{w}^{*}\|_{\mathbf{L}^{2}(\Omega)} +2\mu\|\varepsilon(\mathbf{w})\|_{\underset{\sim}{\mathbf{L}}^2(\Omega)}\|\varepsilon(\mathbf{w}^{*})\|_{\underset{\sim}{\mathbf{L}}^2(\Omega)},
\end{equation*}
which together with \eqref{s2.11} and \eqref{s2.13} yields
\begin{equation}\label{s2.14}
\lambda\|div\mathbf{w}\|_{0,\Omega}\leq C_{\Omega}\|\mathbf{f}\|_{\mathbf{L}^{2}(\Omega)}.
\end{equation}
From \eqref{s2.11} and \eqref{s2.14} we obtain
\begin{equation*}
\|\mathbf{w}\|_{\mathbf{H}^{1}(\Omega)}+\lambda\|div\mathbf{w}\|_{0,\Omega}\leq C_{\Omega}\|\mathbf{f}\|_{\mathbf{L}^{2}(\Omega)}.
\end{equation*}
\indent By Lemma 2.1, there exists $\Phi\in\mathbf{H}^{1+t}(\Omega)\cap\mathbf{H}_{0}^{1}(\Omega)$ such that
\begin{align}
&div\Phi=div\mathbf{w},\nonumber\\
&\|\Phi\|_{\mathbf{H}^{1+t}(\Omega)}\leq C_{\Omega}\|div\mathbf{w}\|_{t,\Omega}.\label{s2.15}
\end{align}
From \eqref{s2.11} and \eqref{s2.15} we get
\begin{equation}\label{s2.16}
\|\Phi\|_{\mathbf{H}^{1+t}(\Omega)}\leq C_{\Omega}(|div\mathbf{w}|_{t,\Omega}+\|\mathbf{f}\|_{\mathbf{L}^{2}(\Omega)}).
\end{equation}
The equation corresponding to \eqref{s2.4} states as:
\begin{equation}\label{s2.17}
-\mu\Delta\mathbf{w}-(\mu+\lambda)\nabla(div\mathbf{w})=\mathbf{f}.
\end{equation}
Define
\begin{equation}\label{s2.18}
\mathbf{w}'=\mathbf{w}-\Phi,~~~~~~~~~g=-(\frac{\mu+\lambda}{\mu})div\mathbf{w},
\end{equation}
then $(\mathbf{w}',g)$ satisfies the following Stokes equation
\begin{equation}\label{s2.18r}
-\Delta\mathbf{w}'+\nabla g=\mathbf{F},~~~div\mathbf{w}'=0,
\end{equation}
where $\mathbf{F}=\frac{1}{\mu}\mathbf{f}+\Delta\Phi$ and $\Delta\Phi\in \mathbf{H}^{-1+t}(\Omega)\subset \mathbf{H}^{-1+s}(\Omega)$.\\
By Theorem 7 in \cite{savare} and the closed graph theorem (see also page 847 in \cite{farhloul})  we have $(\mathbf{w}',g)\in \mathbf{H}^{1+s}(\Omega)\times H^{s}(\Omega)$ with the estimate
\begin{equation}\label{s2.19}
\|\mathbf{w}'\|_{\mathbf{H}^{1+s}(\Omega)}+\|g\|_{s,\Omega}\leq C\|\frac{1}{\mu}\mathbf{f}+\Delta\Phi\|_{\mathbf{H}^{-1+s}(\Omega)}
\end{equation}
where $s<\frac{1}{2}$ and $s$ can be close to $\frac{1}{2}$ arbitrarily, thus we get $\mathbf{w}=\mathbf{w}'+\Phi\in \mathbf{H}^{1+s}(\Omega)$.

Substituting \eqref{s2.18} into \eqref{s2.19} and applying \eqref{s2.16} yield
\begin{equation}\label{s2.20}
\|\mathbf{w}\|_{\mathbf{H}^{1+s}(\Omega)}+\frac{\mu+\lambda}{\mu}|div\mathbf{w}|_{s,\Omega}
\leq C_{\Omega}(\|\mathbf{f}\|_{\mathbf{L}^{2}(\Omega)}+
|div\mathbf{w}|_{s,\Omega}).
\end{equation}
Let $\lambda_{0}=2C_{\Omega}\mu_{1}$ where $C_{\Omega}$ is the constant in \eqref{s2.20}.
For $\lambda>\lambda_{0}$, we obtain from \eqref{s2.20} that
\begin{equation}\label{s2.21}
\|\mathbf{w}\|_{\mathbf{H}^{1+s}(\Omega)}+\frac{\lambda}{2\mu_{1}}|div\mathbf{w}|_{s,\Omega}
\leq C_{\Omega}\|\mathbf{f}\|_{\mathbf{L}^{2}(\Omega)},
\end{equation}
which implies \eqref{s2.8} for $\lambda>\lambda_{0}$. When $\lambda\leq\lambda_{0}$, the conclusion follows directly from the standard elliptic regularity estimate for the problem.
\end{proof}

In the proof of Theorem 2.1, in (\ref{s2.18r}) we use the result that the right-hand side of Stokes equation $\mathbf{F}=\frac{1}{\mu}\mathbf{f}+\Delta\Phi\in \mathbf{H}^{-1+s}(\Omega)~(s<\frac{1}{2})$ to get $\mathbf{w}'\in  \mathbf{H}^{1+s}(\Omega)$, then we derive (\ref{s2.8}).
But in fact, $\mathbf{F}=\frac{1}{\mu}\mathbf{f}+\Delta\Phi\in \mathbf{H}^{-1+t}(\Omega)~(t>\frac{1}{2})$.
In addition, from \S6.2 in \cite{grisvard2} we know that when the right-hand side $\mathbf{F}\in \mathbf{L}^{2}(\Omega)$, the generalized solution of Stokes equation $\mathbf{w}'\in  \mathbf{H}^{1+t}(\Omega)$ and $g\in H^{t}(\Omega)~(t>\frac{1}{2})$.
Thus, by interpolation of Sobolev space (see for instance \cite{brenner}), when $\mathbf{F}=\frac{1}{\mu}\mathbf{f}+\Delta\Phi\in \mathbf{H}^{-1+t}(\Omega)~(t>\frac{1}{2})$,
we have $\mathbf{w}'\in  \mathbf{H}^{-1+t'}(\Omega)~(t'>\frac{1}{2})$.
Therefore, we think the following regularity assumption is reasonable:

{\bf $\mathbf{R}(\Omega)$.} For any $\mathbf{f}\in \mathbf{L}^{2}(\Omega)$, there exists $\mathbf{w}\in \mathbf{H}^{1+s}(\Omega)\cap \mathbf{W}^{2,p}(\Omega)\cap \mathbf{H}^{1}_{0}(\Omega)$ satisfying
\begin{equation*}
a(\mathbf{w}, \mathbf{v})=b(\mathbf{f}, \mathbf{v}),~~~\forall \mathbf{v}\in \mathbf{H}^{1}_{0}(\Omega),
\end{equation*}
and
\begin{equation*}
\|\mathbf{w}\|_{\mathbf{H}^{1+s}(\Omega)}+\lambda \|div\mathbf{w}\|_{s,\Omega}\leq C_{\Omega}\|\mathbf{f}\|_{\mathbf{L}^{2}(\Omega)},
\end{equation*}
for some $\frac{1}{2}-\varepsilon<s\leq 1$ where $\varepsilon>0$ is an arbitrarily small constant and $p=\frac{2}{2-s}$.

\section{The nonconforming C-R element approximation}\label{tensor-models}

Assume that $\pi_{h}=\{\kappa\}$ is a regular triangulation of $\Omega$ with mesh-size function $h(\mathbf{x})$ whose value is the diameter $h_{\kappa}$ of the element $\kappa$ containing $\mathbf{x}$, $\frac{h_{\kappa}}{\rho_{\kappa}}\leq \nu$ with $\rho_{\kappa}$ the supremum of diameter of circle contained in $\kappa$ (see (17.1) in \cite{ciarlet}), and $h=\max \limits_{\mathbf{x}\in\Omega}h(\mathbf{x})$ is the mesh diameter of $\pi_{h}$.
Let $\mathcal{E}_{h}$ denote the set of all edges of elements $\kappa\in\pi_{h}$. We split this set as $\mathcal{E}_{h}
 =\mathcal{E}_{h}^{i}\cup\mathcal{E}_{h}^{b}$ with $\mathcal{E}_{h}^{i}$ and $\mathcal{E}_{h}^{b}$ being the sets of inner and boundary edges, respectively. Let $S_0^{h}(\Omega)$ be the C-R element space defined on $\pi_{h}$:
\begin{equation*}
S_{0}^{h}(\Omega)=\{v\in L^{2}(\Omega): v\mid_{\kappa} \in P_{1}(\kappa)~~\forall  \kappa\in \pi_{h}, \int_{\ell}[[v]]ds=0~~\forall \ell\in \mathcal{E}_{h}^{i},~~\int_{\ell}vds=0~~\forall \ell\in \mathcal{E}_{h}^{b}\},
\end{equation*}
where $[[\cdot]]$ is the jump across an edge $\ell\in \mathcal{E}_{h}$ defined as follows.

 If $\ell\in \mathcal{E}_{h}^{i}$ is shared by two elements $\kappa_1$ and $\kappa_2$ in $\pi_h$, and $v_i=v|_{\kappa_i} (i=1,2)$, then
 $[[v]]=(v_{1}-v_{2})|_{\ell}$; If $\ell\in \mathcal{E}_{h}^{b}$, then $[[v]]=v|_{\ell}$.

Denote $\mathbf{S}_{0}^{h}(\Omega)=S_{0}^{h}(\Omega)\times S_{0}^{h}(\Omega)$, and define
\begin{equation*}
\mathbf{H}(h)=\mathbf{S}_{0}^{h}(\Omega) +\mathbf{H}_{0}^{1}(\Omega)=\{ \mathbf{w}_h+
\mathbf{w}: \mathbf{w}_h\in \mathbf{S}_{0}^h(\Omega),  \mathbf{w}\in \mathbf{H}_{0}^{1}(\Omega)\}.
\end{equation*}
Denote
\begin{equation}\label{s3.1}
a_{h}(\mathbf{u},\mathbf{v})=\mu\int _{\Omega }\mathbf{\nabla}_{h}\mathbf{u}:\mathbf{\nabla}_{h}\mathbf{v}d\mathbf{x}+(\mu+\lambda)\int_{\Omega} (div_{h}\mathbf{u})(div_{h}\mathbf{v})d\mathbf{x},~~~\forall \mathbf{u},\mathbf{v}\in\mathbf{H}(h),%\label{s3.1}
\end{equation}
where $(\mathbf{\nabla}_{h}\mathbf{v})|_\kappa=\mathbf{\nabla}(\mathbf{v}|_\kappa)$ and $(div_{h}\mathbf{v})|_\kappa=div(\mathbf{v}|_\kappa)$ for any $\mathbf{v}\in\mathbf{H}(h)$. It is easy to know that $a_{h}(\cdot,\cdot)$ is continuous and positive definite in $\mathbf{H}(h)$.

Define the nonconforming energy norm $\|\cdot\|_{h}$ on $\mathbf{H}(h)$ by
\begin{equation*}
\|\mathbf{v}\|_{h}=\sqrt{a_{h}(\mathbf{v},\mathbf{v})},
\end{equation*}
and denote
\begin{equation*}
|\mathbf{v}|_{1,h}=\sqrt{\sum\limits_{\kappa\in\pi_{h}}|\mathbf{v}|_{\mathbf{H}^1(\kappa)}^2}.
\end{equation*}
From the Poincar$\acute{e}$-Friedrichs inequality (cf. \cite{brenner2003}) we know that $|\cdot|_{1,h}$ is also a norm on $\mathbf{H}(h)$, and a simple calculation shows that
\begin{equation*}
|\mathbf{v}|_{1,h}^2=\sum\limits_{\kappa\in\pi_{h}}|\mathbf{v}|_{\mathbf{H}^1(\kappa)}^2=\sum\limits_{\kappa\in\pi_{h}}\int_\kappa \mathbf{\nabla}\mathbf{v}:\mathbf{\nabla}\mathbf{v}d\mathbf{x}\leq C \|\mathbf{v}\|_{h}^2.
\end{equation*}
Define the C-R element interpolation operator $\mathbf{I}_{h}$ : $\mathbf{H}_{0}^{1}(\Omega) \to \mathbf{S}_{0}^{h}(\Omega)$ by:
\begin{equation*}
\int_{\ell}\mathbf{I}_{h}\mathbf{v} ds=\int_{\ell}\mathbf{v} ds,~~~ \forall~\ell\in \mathcal{E}_{h}.
\end{equation*}
The C-R nonconforming finite element discretization of \eqref{s2.2} is as follows: Find $(\omega_h, \mathbf{u}_h)\in \mathbb{R}\times \mathbf{S}_{0}^{h}(\Omega)$ with $\|\mathbf{u}_h\|_{h}=1$ such that
\begin{equation}\label{s3.2}
a_{h}(\mathbf{u}_h,\mathbf{v}_h)=\omega_h b(\mathbf{u}_h,\mathbf{v}_h),~~~\forall \mathbf{v}_h\in \mathbf{S}_{0}^{h}(\Omega).
\end{equation}
\indent The source problem associated with \eqref{s3.2} states as: Find $\mathbf{w_h}\in \mathbf{S}_{0}^{h}(\Omega)$ such that
\begin{equation}\label{s3.3}
a_{h}(\mathbf{w_h},\mathbf{v})= b(\mathbf{f},\mathbf{v}),~~~\forall \mathbf{v}\in \mathbf{S}_{0}^{h}(\Omega).
\end{equation}
The well-posedness of \eqref{s3.3} has also been discussed in \cite{brenner}.

Let $\mathbf{w}$  be the solution of \eqref{s2.4}. Define the
consistency term: For any $\mathbf{v}\in \mathbf{H}(h)$,
\begin{equation}\label{s3.4}
D_{h}(\mathbf{w},\mathbf{v})=a_{h}(\mathbf{w},\mathbf{v})-b(\mathbf{f}, \mathbf{v}).
\end{equation}
\indent To estimate the consistency term, we need the following trace inequalities.
\begin{lemma}\label{lem3.1}
For any $\kappa\in \pi_{h}$ and $w\in H^{1+s}(\kappa)$, the following trace inequalities hold:
\begin{align*}
&\|w\|_{0,\partial\kappa}
\leq C(h_{\kappa}^{-\frac{1}{2}}\|w\|_{0,\kappa}+h_{\kappa}^{\frac{1}{2}}|w|_{1,\kappa}),\\
&\|\nabla w\|_{0,\partial\kappa}
\leq C(h_{\kappa}^{-\frac{3}{2}}\| w\|_{0,\kappa}+h_{\kappa}^{-\frac{1}{2}}| w|_{1,\kappa}+h_{\kappa}^{s-\frac{1}{2}}|w|_{1+s,\kappa})\quad(\frac{1}{2}\leq s\leq 1).
\end{align*}
 \end{lemma}
\begin{proof}
 The conclusion is followed by using the trace theorem on the reference element and the scaling argument.
\end{proof}

\begin{lemma}\label{lem3.2}
Let $\ell\subset \partial\kappa$ be an edge of element $\kappa$. For any $g\in H^{\frac{1}{2}-r}(\ell)$, there exists a lifting $v_{g}$ of $g$ such that $ v_{g}\in H^{1-r}(\kappa)$ ($0<r<\frac{1}{2}$), $v_{g}|_{\ell}=g$, $v_{g}|_{\partial\kappa\backslash \ell}=0$ and
\begin{equation*}
|v_{g}|_{1-r,\kappa}+h_{\kappa}^{r-1}\|v_{g}\|_{0,\kappa}\leq C h_{\kappa}^{-\delta}\|g\|_{\frac{1}{2}-r,\partial\kappa}.
\end{equation*}
where $C$ depends on the constant $\nu$ in regular triangulation but is independent of $\lambda$,
and $\delta=\frac{1}{2}-r$.
\end{lemma}
\begin{proof}
Let $\widehat{\kappa}$ denote the reference element, introduce the affine mappings $\widehat{x}\to F_{\kappa}(\widehat{x})=B_{\kappa}\widehat{x}+b_{\kappa}$ which maps the reference element $\widehat{\kappa}$ on $\kappa$ and
 $\widehat{x}\to B_{\ell}\widehat{x}+b_{\ell}$  which maps the reference edge $\widehat{\ell}$ on an edge $\ell$ of $\kappa$.
 Then, from \cite{ciarlet2013} we have
 $$|det B_{\kappa}|\leq C h_{\kappa}^{2},\quad\|B_{\kappa}\|\leq C h_{\kappa},\quad|det B_{\kappa}|^{-1}\leq C \rho_{\kappa}^{-2},\quad
 \|B_{\kappa}^{-1}\|\leq C \rho_{\kappa}^{-1},\quad|det B_{\ell}|\leq C h_{\kappa},$$
 where $\|\cdot\|$ stands for the Euclidean norm of matrix.

 From Theorem 1.5.2.3 in \cite{grisvard} we know that any $\widehat{g}\in H^{\frac{1}{2}-r}(\widehat{\ell})$ can be extended to be a function belonging to $\mathbf{H}^{\frac{1}{2}-s}(\partial\widehat{\kappa})$ through
the trivial extension by zero to all of $\partial\widehat{\kappa}$.
Thanks to the inverse trace theorem (see page 387 in \cite{kufner}, or page 1767 in \cite{caiz}) we know that there exists a lifting $\widehat{v}_{g}$ of $\widehat{g}$ such that
$ \widehat{v}_{g}\in H^{1-r}(\widehat{\kappa})$, $\widehat{v}_{g}|_{\partial\widehat{\kappa}}=\widehat{g}$ and
\begin{equation}\label{exms1}
\|\widehat{v}_{g}\|_{1-r,\widehat{\kappa}}\leq C\|\widehat{g}\|_{\frac{1}{2}-r,\partial\widehat{\kappa}} =
C\|\widehat{g}\|_{\frac{1}{2}-r,\widehat{\ell}}.
\end{equation}
From the relationships between the seminorms on affine equivalent elements in Sobolev space
(see, e.g., \cite{ciarlet,ciarlet2013}) we deduce that
\begin{align}\label{exms11}
&h_{\kappa}^{r-1}\|v_{g}\|_{0,\kappa}\leq Ch_{\kappa}^{r-1}|det B_{\kappa}|^{\frac{1}{2}}\|\widehat{v}_{g}\|_{0,\widehat{\kappa}}
\leq Ch_{\kappa}^{r-1}h_{\kappa}\|\widehat{v}_{g}\|_{0,\widehat{\kappa}}=Ch_{\kappa}^{r}\|\widehat{v}_{g}\|_{0,\widehat{\kappa}},\\\label{exms12}
&|v_{g}|_{1-r,\kappa}\leq \|B_{\kappa}^{-1}\|^{1-r}|det B_{\kappa}|^{\frac{1}{2}}|\widehat{v}_{g}|_{1-r,\widehat{\kappa}}
\leq (\frac{1}{\rho_{\kappa}})^{1-r} h_{\kappa}|\widehat{v}_{g}|_{1-r,\widehat{\kappa}},\\\label{exms13}
&\|\widehat{g}\|_{0,\widehat{\ell}}\leq C|det B_{\ell}|^{-\frac{1}{2}}\|{g}\|_{0,\ell}
\leq C \rho_{\kappa}^{-\frac{1}{2}}|g|_{0,\ell},\\\label{exms14}
&|\widehat{g}|_{\frac{1}{2}-r,\widehat{\ell}}\leq C\|B_{\ell}\|^{\frac{1}{2}-r}|det B_{\ell}|^{-\frac{1}{2}}|{g}|_{\frac{1}{2}-r,\ell}
\leq C h_{\ell}^{\frac{1}{2}-r} \rho_{\kappa}^{-\frac{1}{2}}|g|_{\frac{1}{2}-r,\ell}.
\end{align}
Since $\frac{h_{\kappa}}{\rho_{\kappa}}\leq \nu$, we have $\rho_{\kappa}\geq \frac{h_{\kappa}}{\nu}$.
Thus, from (\ref{exms11}), (\ref{exms12}) and (\ref{exms1}) we deduce
\begin{align*}
&h_{\kappa}^{r-1}\|v_{g}\|_{0,\kappa}+|v_{g}|_{1-r,\kappa}\leq C(h_{\kappa}^{r} \|\widehat{v}_{g}\|_{0,\widehat{\kappa}}+ (\frac{\nu}{h_{\kappa}})^{1-r} h_{\kappa}|\widehat{v}_{g}|_{1-r,\widehat{\kappa}}\\
&\quad\quad\leq C\max\{1,\nu^{1-r}\}h_{\kappa}^{r}\|\widehat{v}_{g}\|_{1-r,\widehat{\kappa}}\leq C\nu^{1-r} h_{\kappa}^{r}\|\widehat{g}\|_{\frac{1}{2}-r,\widehat{\ell}},
\end{align*}
and from (\ref{exms13}) and (\ref{exms14}) we derive
\begin{equation*}
\|\widehat{g}\|_{0,\widehat{\ell}}+|\widehat{g}|_{\frac{1}{2}-r,\widehat{\ell}}\leq C\rho_{\kappa}^{-\frac{1}{2}}|g|_{0,\ell}
+Ch_{\ell}^{\frac{1}{2}-r} \rho_{\kappa}^{-\frac{1}{2}}|g|_{\frac{1}{2}-r,\ell}
\leq C\nu^{\frac{1}{2}}\max\{h_{\kappa}^{-r}, h_{\kappa}^{-\frac{1}{2}}\}\|g\|_{\frac{1}{2}-r,\ell}.
\end{equation*}
Combining the above two estimates, we get the desired result.
\end{proof}

\begin{lemma}\label{lem3.3}
Let $\mathbf{w}$ be the solution of (\ref{s2.4}),
and $\mathbf{w}\in\mathbf{H}^{1+r}(\Omega)\cap \mathbf{W}^{2,p}(\Omega)$ $(0<r<\frac{1}{2}$, $p=\frac{2}{2-r})$,
then
\begin{equation}\label{s3.5}
\|\mu\nabla\mathbf{w}\gamma
+(\lambda+\mu)div\mathbf{w} \gamma\|_{\mathbf{H}^{r-\frac{1}{2}}(\ell)}
\leq  C h_{\kappa}^{-\delta}( h_{\ell}^{1-r}  \|\mathbf{f}\|_{\mathbf{L}^{2}(\kappa)}
+\mu\|\nabla\mathbf{w}\|_{\underset{\sim}{\mathbf{H}}^r(\kappa)}
+(\lambda+\mu)\|div\mathbf{w}\|_{r,\kappa}),~~~\forall \kappa\in
\pi_{h},~\ell\subset\partial\kappa,
\end{equation}
where $\gamma$ is the unit out normal to $\partial\kappa$, $C$ depends on the constant $\nu$ in regular triangulation but is independent of $\lambda$ and $\delta=\frac{1}{2}-r$.
\end{lemma}
\begin{proof}
We use the proof method of Corollary 3.3 on page 1384
in \cite{bernardi} or Lemma 2.1 in \cite{caiz} to prove \eqref{s3.5}.

First, we shall prove that the following Green's formula
\begin{equation}\label{s3.6}
\int\limits_{\partial\kappa}(\nabla\mathbf{w}\gamma)\cdot\mathbf{v}ds
=\int\limits_{\kappa}\Delta\mathbf{w}\cdot\mathbf{v}d\mathbf{x}+\int\limits_{\kappa}\nabla\mathbf{w}:\nabla \mathbf{v}d\mathbf{x},~~~\forall \kappa\in\pi_{h}
\end{equation}
holds for all $\mathbf{v}\in \mathbf{H}^{1-r}(\kappa)$
with $0<r<\frac{1}{2}$.

Let $\mathbf{H}^{-r}(\kappa)$ be the dual of $\mathbf{H}_{0}^{r}(\kappa)$ which is the closure of $\mathbf{C}_{0}^{\infty}(\kappa)$ in $\mathbf{H}^{r}(\kappa)$ norm.
Since  $\mathbf{H}^{r}(\kappa)$ is the same space as $\mathbf{H}_{0}^{r}(\kappa)$ for $r\in (0,\frac{1}{2})$ (see, e.g., Theorem 1.4.2.4 in \cite{grisvard}) and $\nabla\mathbf{v}$ is in $\mathbf{H}^{-r}(\kappa)$, the term $\int_{\kappa}\nabla\mathbf{w}:\nabla\mathbf{v}dx$ in (\ref{s3.6}) then can be viewed as a duality pair between $\mathbf{H}^{r}(\kappa)$ and $\mathbf{H}^{-r}(\kappa)$.
By the Sobolev imbedding theorem we get
$\mathbf{H}^{1-r}(\kappa)\hookrightarrow \mathbf{L}^{\frac{2}{r}}(\kappa)$ continuously,
thus
the term $\int_{\kappa}\Delta\mathbf{w}\cdot\mathbf{v}d\mathbf{x}$ in (\ref{s3.6}) can be viewed as a duality pair between $\mathbf{L}^{p}(\kappa)$ and $\mathbf{L}^{\frac{2}{r}}(\kappa)$. Since  $\mathbf{w}\in \mathbf{W}^{2,p}(\Omega)$ $(1<p<2)$ is the solution of (\ref{s2.4}), there is $\tau>0$ such that
$\mathbf{w}\in\mathbf{W}^{2,p+\tau}(\Omega)$.
By the trace theorem, there is $\tau_{1}>0$ which can be arbitrarily close to $0$
such that $ \mathbf{H}^{1-r}(\kappa)\hookrightarrow\mathbf{L}^{\frac{1}{r}-\tau_{1}}(\partial\kappa)$ continuously,
and there is $\tau_{2}>0$ such that
$\nabla \mathbf{w}\gamma|_{\partial\kappa}\in \mathbf{L}^{\frac{1}{1-r}+\tau_{2}}(\partial\kappa)$,
thus $(\nabla\mathbf{w}\gamma)\cdot\mathbf{v}|_{\partial\kappa}\in L^{1}(\partial\kappa)$. To sum up, all terms in (\ref{s3.6}) make sense.

Then, the validity of (\ref{s3.6}) follows from the standard density argument ($\mathbf{C}^{\infty}(\overline{\kappa})$ is dense in $\mathbf{H}^{1-r}(\kappa)$ ) and the fact that (\ref{s3.6}) holds for $\mathbf{C}^{\infty}(\overline{\kappa})$ function $\mathbf{v}$.

Using the same argument as above,  we can deduce  that for all $\mathbf{v}\in \mathbf{H}^{1-r}(\kappa)$
with $0<r<\frac{1}{2}$
\begin{equation}\label{s3.7}
\int\limits_{\partial\kappa}div\mathbf{w} \gamma\cdot\mathbf{v}ds=\int_{\kappa }\mathbf{\nabla}(div\mathbf{w})\cdot\mathbf{v}d\mathbf{x}
+\int_{\kappa }(div\mathbf{w})(div\mathbf{v})d\mathbf{x}.
\end{equation}
By the trace theorem, $\mathbf{v}|_{\partial\kappa}$ is in $\mathbf{H}^{\frac{1}{2}-r}(\partial\kappa)$.
Since, for each  edge $\ell\subset\partial\kappa$, the trivial extension of functions in $\mathbf{H}^{\frac{1}{2}-r}(\ell)$ by zero to all of $\partial\kappa$
belongs to $\mathbf{H}^{\frac{1}{2}-r}(\partial\kappa)$ (see, e.g., Theorem 1.5.2.3 in \cite{grisvard}),
this interpretation enables us to define the duality pair on each edge $\ell$ of $\partial\kappa$
$$\mu\int\limits_{\ell}(\nabla\mathbf{w}\gamma)\cdot\mathbf{v}ds
+(\lambda+\mu)\int\limits_{\ell}div\mathbf{w} \gamma\cdot\mathbf{v}ds
:=<\mu\nabla\mathbf{w}\gamma+(\lambda+\mu)div\mathbf{w} \gamma,\mathbf{v}>_{\ell},$$
where $(\mu\nabla\mathbf{w}\gamma+(\lambda+\mu)div\mathbf{w} \gamma)|_{\ell}\in \mathbf{H}^{r-\frac{1}{2}}(\ell)$ and $\mathbf{v}|_{\ell}\in \mathbf{H}^{\frac{1}{2}-r}(\ell)$.\\
\indent For any $\mathbf{g}\in \mathbf{H}^{\frac{1}{2}-r}(\ell)$,
from Lemma 3.2 we know
that there exists a lifting $\mathbf{v}_{\mathbf{g}}$ of $\mathbf{g}$ such that $\mathbf{v}_{\mathbf{g}}\in \mathbf{H}^{1-r}(\kappa)$, $\mathbf{v}_{\mathbf{g}}|_{\ell}=\mathbf{g}$, $\mathbf{v}_{\mathbf{g}}|_{\partial\kappa\setminus\ell}=0$, and
\begin{eqnarray*}
\|\nabla\mathbf{v}_{\mathbf{g}}\|_{\underset{\sim}{\mathbf{H}}^{-r}(\kappa)}+h_{\kappa}^{r-1}\| \mathbf{v}_{\mathbf{g}}\|_{\mathbf{L}^2(\kappa)}\leq C h_{\kappa}^{-\delta}\|\mathbf{g}\|_{\mathbf{H}^{\frac{1}{2}-r}(\ell)},
\end{eqnarray*}
where $C$ depends on the constant $\nu$ in regular triangulation %(see (17.1) in \cite{ciarlet})
but is independent of $\lambda$.

From Green's formula (\ref{s3.6}) and (\ref{s3.7})
and the definition of the dual norm we deduce
\begin{align}\label{s3.8}
&\int\limits_{\ell}
\mu(\nabla\mathbf{w}\gamma)\cdot\mathbf{g}
+(\lambda+\mu)div\mathbf{w} \gamma\cdot\mathbf{g}ds
=\int\limits_{\partial\kappa}\mu(\nabla\mathbf{w}\gamma)\cdot\mathbf{v}_{\mathbf{g}}+(\lambda+\mu)div\mathbf{w} \gamma\cdot\mathbf{v}_{g}ds\nonumber\\
&=\mu(\int\limits_{\kappa}\Delta\mathbf{w}\cdot\mathbf{v}_{g}d\mathbf{x}+\int\limits_{\kappa}\nabla\mathbf{w}:\nabla \mathbf{v}_{\mathbf{g}}d\mathbf{x})
+(\lambda+\mu)(\int_{\kappa }\mathbf{\nabla}(div\mathbf{w})\cdot\mathbf{v}_{g}d\mathbf{x}
+\int_{\kappa }(div\mathbf{w})(div\mathbf{v}_{g})d\mathbf{x})
\nonumber\\
&=\int\limits_{\kappa}\mathbf{f}\cdot\mathbf{v}_{g}d\mathbf{x}+\mu\int\limits_{\kappa}\nabla\mathbf{w}:\nabla \mathbf{v}_{\mathbf{g}}d\mathbf{x}
+(\lambda+\mu)\int_{\kappa }(div\mathbf{w})(div\mathbf{v}_{g})d\mathbf{x}.
\nonumber\\
&\leq C(\|\mathbf{f}\|_{\mathbf{L}^{2}(\kappa)}\|\mathbf{v}_{g}\|_{\mathbf{L}^{2}(\kappa)}+\mu\|\nabla\mathbf{w}\|_{\underset{\sim}{\mathbf{H}}^r(\kappa)}\| \nabla\mathbf{v}_{\mathbf{g}}\|_{\underset{\sim}{\mathbf{H}}^{-r}(\kappa)}
+(\lambda+\mu)\|div\mathbf{w}\|_{r,\kappa}\|div\mathbf{v}_{\mathbf{g}}\|_{-r,\kappa})\nonumber\\
&\leq C{\color{blue}h_{\kappa}^{-\delta}}( h_{\ell}^{1-r}  \|\mathbf{f}\|_{\mathbf{L}^{2}(\kappa)}
+\mu\|\nabla\mathbf{w}\|_{\underset{\sim}{\mathbf{H}}^r(\kappa)}
+(\lambda+\mu)\|div\mathbf{w}\|_{r,\kappa})
     \|\mathbf{g}\|_{\mathbf{H}^{\frac{1}{2}-r}(\ell)},
\end{align}
by the definition of the dual norm we have
\begin{equation*}
\|\mu\nabla\mathbf{w}\gamma
+(\lambda+\mu)div\mathbf{w} \gamma\|_{\mathbf{H}^{r-\frac{1}{2}}(\ell)}
=\sup\limits_{\mathbf{g}\in \mathbf{H}^{\frac{1}{2}-r}(\ell)}\frac{|\int\limits_{\ell}
\mu(\nabla\mathbf{w}\gamma)\cdot\mathbf{g}
+(\lambda+\mu)div\mathbf{w} \gamma\cdot\mathbf{g}ds|}
{\|\mathbf{g}\|_{\mathbf{H}^{\frac{1}{2}-r}(\ell)}}.
\end{equation*}
Combining the above two relationships we obtain \eqref{s3.5}.
\end{proof}

Based on the standard argument (see, e.g., \cite{brenner}), the following consistency error estimate can be proved.
\begin{theorem}\label{thm3.1}
Let $\mathbf{w}\in \mathbf{H}^{1+s}(\Omega)$ be the solution of \eqref{s2.4} and suppose that
$\mathbf{R}(\Omega)$ holds, then
\begin{equation}\label{s3.9}
|D_{h}(\mathbf{w},\mathbf{v})|\leq C h^{s}\|\mathbf{f}\|_{\mathbf{L}^{2}(\Omega)}\|\mathbf{v}\|_{h},~~~\forall \mathbf{v}\in
\mathbf{H}(h).
\end{equation}
\end{theorem}
\begin{proof}
 Using integration by parts, we get
\begin{equation}\label{s3.10}
\int _{\Omega }\mathbf{\nabla}\mathbf{w}:\mathbf{\nabla}_h\mathbf{v}d\mathbf{x}+\int _{\Omega }\triangle\mathbf{w}\cdot\mathbf{v}d\mathbf{x}
=\sum\limits_{\ell\in\mathcal{E}_{h}}\int _{\ell}\frac{\partial\mathbf{w}}{\partial\gamma}\cdot[[\mathbf{v}]]ds,
\end{equation}
\begin{align}
\int_{\Omega }(div\mathbf{w})(div_h\mathbf{v})d\mathbf{x}+\int_{\Omega }\mathbf{\nabla}(div\mathbf{w})\cdot\mathbf{v}d\mathbf{x}
=\sum\limits_{\ell\in\mathcal{E}_{h}}\int\limits_{\ell}div\mathbf{w} \gamma\cdot[[\mathbf{v}]]ds.\label{s3.11}
\end{align}
Combining \eqref{s2.17}, \eqref{s3.1}, \eqref{s3.10} and \eqref{s3.11}, we deduce
\begin{align}\label{s3.12}
&|a_{h}(\mathbf{w},\mathbf{v})-\int_{\Omega}\mathbf{f}\cdot\mathbf{v}d\mathbf{x}|=|a_{h}(\mathbf{w},\mathbf{v})
-\int_{\Omega}(-\mu\Delta\mathbf{w}-(\mu+\lambda)\nabla div\mathbf{w})\cdot\mathbf{v}d\mathbf{x}|\nonumber\\
&=\mu
\sum\limits_{\ell\in\mathcal{E}_{h}}\int _{\ell}\frac{\partial\mathbf{w}}{\partial\gamma}\cdot[[\mathbf{v}]]ds+
(\mu+\lambda)\sum\limits_{\ell\in\mathcal{E}_{h}}\int\limits_{\ell}div\mathbf{w} \gamma\cdot[[\mathbf{v}]]ds.
\end{align}
For $\ell\in\mathcal{E}_{h}$, $\kappa\in\pi_{h}$, define
\begin{equation*}
P_{\ell}\mathbf{f}=\frac{1}{|\ell|}\int\limits_{\ell}\mathbf{f}ds,
~~~P_{\kappa}\mathbf{f}=\frac{1}{|\kappa|}\int\limits_{\kappa}\mathbf{f}d\mathbf{x}.\\
\end{equation*}
Suppose that $\kappa_{1},\kappa_{2}\in \pi_{h}$ such that
$\kappa_{1}\cap \kappa_{2}=\ell$. Since $[[\mathbf{v}]]$ is a linear function vanishing
at the midpoint of $\ell$, we have
\begin{subequations}
\begin{align}
&|\int_{\ell}\frac{\partial \mathbf{w}}{\partial \gamma}\cdot[[\mathbf{v}]]ds|
=|\int_{\ell}(\frac{\partial \mathbf{w}}{\partial
\gamma}-P_{\ell}(\frac{\partial \mathbf{w}}{\partial \gamma}))\cdot[[\mathbf{v}]]ds| \nonumber\\
& \quad \quad=|\int_{\ell}(\frac{\partial\mathbf{w}}{\partial \gamma}-P_{\ell}(\frac{\partial \mathbf{w}}{\partial
\gamma}))\cdot([[\mathbf{v}]]-P_{\ell}[[\mathbf{v}]])ds|\label{s3.13a}\\
&\quad \quad = |\int_{\ell}\frac{\partial \mathbf{w}}{\partial \gamma}\cdot([[\mathbf{v}]]-P_{\ell}[[\mathbf{v}]])ds|,\label{s3.13b}
\end{align}
\end{subequations}
\begin{subequations}
\begin{align}
&|\int_{\ell}div\mathbf{w}\gamma\cdot[[\mathbf{v}]]ds|
=|\int_{\ell}(div\mathbf{w}\gamma-P_{\ell}(div\mathbf{w}\gamma))\cdot[[\mathbf{v}]]ds| \nonumber\\
& \quad \quad=|\int_{\ell}(div\mathbf{w}\gamma-P_{\ell}(div\mathbf{w}\gamma)\cdot([[\mathbf{v}]]-P_{\ell}[[\mathbf{v}]])ds|\label{s3.14a}\\
&\quad \quad = |\int_{\ell}div\mathbf{w}\gamma\cdot([[\mathbf{v}]]-P_{\ell}[[\mathbf{v}]])ds|.\label{s3.14b}
\end{align}
\end{subequations}

Then, when $s\in [ \frac{1}{2},1]$, using (\ref{s3.13a}) and Schwarz inequality we deduce
\begin{align}
|\int_{\ell}\frac{\partial \mathbf{w}}{\partial \gamma}\cdot[[\mathbf{v}]]ds|
&\leq \sum\limits_{i=1,2}\|\nabla \mathbf{w} \gamma-
P_{\ell}(\nabla \mathbf{w}
\gamma)\|_{\mathbf{L}^{2}(\ell)}\|\mathbf{v}|_{\kappa_{i}}-P_{\ell}(\mathbf{v}|_{\kappa_{i}})\|_{\mathbf{L}^{2}(\ell)} \nonumber\\
& \leq\sum\limits_{i=1,2}\|\nabla (\mathbf{w}-\mathbf{I}_{h}\mathbf{w}) \gamma\|_{\mathbf{L}^{2}(\ell)}\|\mathbf{v}|_{\kappa_{i}}-P_{\kappa_{i}}(\mathbf{v}|_{\kappa_{i}})\|_{\mathbf{L}^{2}(\ell)},\label{s3.15}
\end{align}
and by Lemma 3.1 and the standard error estimates for
$L^{2}$-projection we get
\begin{align}
\|\nabla (\mathbf{w}-\mathbf{I}_{h}\mathbf{w}) \gamma\|_{\mathbf{L}^{2}(\ell)}&\leq C h^{s-\frac{1}{2}}||\mathbf{w}||_{\mathbf{H}^{1+s}(\kappa_{i})},\nonumber\\
\|\mathbf{v}|_{\kappa_{i}}-P_{\kappa_{i}}(\mathbf{v}|_{\kappa_{i}})\|_{\mathbf{L}^{2}(\ell)}&\leq C h^{\frac{1}{2}}||\mathbf{v}||_{\mathbf{H}^{1}(\kappa_{i})}.\nonumber
\end{align}
Substituting the above two estimates into \eqref{s3.15}, we obtain
\begin{equation}\label{s3.16}
|\int_{\ell}\frac{\partial \mathbf{w}}{\partial \gamma}\cdot[[\mathbf{v}]]ds| \leq C\sum\limits_{i=1,2}h^{s}\|\mathbf{w}\|_{\mathbf{H}^{1+s}(\kappa_{i})}\|\mathbf{v}\|_{\mathbf{H}^{1}(\kappa_{i})}.
\end{equation}
Using the same argument as above,  we can derive that for $s\in [\frac{1}{2},1]$,
\begin{align}
|\int\limits_{\ell}div\mathbf{w} \gamma\cdot[[\mathbf{v}]]ds|\leq C \sum\limits_{i=1,2}h^{s}|div\mathbf{w}|_{s,\kappa_{i}}\|\mathbf{\nabla}_h\mathbf{v}\|_{\underset{\sim}{\mathbf{L}}^2(\kappa_{i})}.\label{s3.17}
\end{align}
Combining \eqref{s3.12}, \eqref{s3.16}, \eqref{s3.17} and \eqref{s2.8}, we deduce
\begin{align*}
&|a_{h}(\mathbf{w},\mathbf{v})-\int_{\Omega}\mathbf{f}\cdot\mathbf{v}d\mathbf{x}|
\leq  C h^{s}\|\mathbf{\nabla}_h\mathbf{v}\|_{\underset{\sim}{\mathbf{L}}^2(\Omega)}\{\mu|\mathbf{w}|_{\mathbf{H}^{1+s}(\Omega)}+(\mu+\lambda)|div\mathbf{w}|_{s,\Omega}\}
\nonumber\\
&\quad\quad\quad \leq C h^{s}\|\mathbf{f}\|_{\mathbf{L}^{2}(\Omega)}\|\mathbf{v}\|_{h},~~~\forall \mathbf{v}\in\mathbf{H}(h).%\label{s3.17}
\end{align*}
Then  \eqref{s3.9} is valid for $s\in [\frac{1}{2},1]$.

When $s<\frac{1}{2}$, from $\mathbf{R}(\Omega)$ we also have $\mathbf{w}\in \mathbf{H}^{1+r}(\Omega)$ by taking $r=s+\frac{0.5-s}{2}$.
Thus, from (\ref{s3.13b}), (\ref{s3.14b}) and Lemma 3.3 we deduce that
\begin{align}\label{s3.18}
&|\int_{\ell}(\mu\frac{\partial \mathbf{w}}{\partial \gamma}+(\mu+\lambda)div \mathbf{w}\gamma)\cdot[[\mathbf{v}]]ds|=
|\int_{\ell}(\mu\frac{\partial \mathbf{w}}{\partial \gamma}+(\mu+\lambda)div \mathbf{w}\gamma)\cdot([[\mathbf{v}]]-P_{\ell}[[\mathbf{v}]])ds|\nonumber\\
&\quad\quad\leq C h_{\kappa}^{-\delta}( h_{\ell}^{1-r}  \|\mathbf{f}\|_{\mathbf{L}^{2}(\kappa)}
+\mu\|\mathbf{w}\|_{\mathbf{H}^{1+r}(\kappa)}
+(\mu+\lambda)\|div\mathbf{w}\|_{r,\kappa})
     \|[[\mathbf{v}]]-P_{\ell}[[\mathbf{v}]]\|_{\mathbf{H}^{\frac{1}{2}-r}(\ell)}.
\end{align}
By using inverse estimate, Lemma 3.1 and the error estimate of $L^{2}$-projection, we derive
\begin{equation*}
\|[[\mathbf{v}]]-P_{\ell}[[\mathbf{v}]]\|_{\mathbf{H}^{\frac{1}{2}-r}(\ell)}\leq C h_{\ell}^{r-\frac{1}{2}}\|[[\mathbf{v}]]-P_{\ell}[[\mathbf{v}]]\|_{\mathbf{L}^{2}(\ell)}
\leq C \sum\limits_{i=1,2}h_{\kappa_{i}}^{r}|\mathbf{v}|_{\mathbf{H}^{1}(\kappa_{i})}.
\end{equation*}
Substituting the above estimate into \eqref{s3.18}, we obtain
\begin{equation*}
|\int_{\ell}(\mu\frac{\partial \mathbf{w}}{\partial \gamma}+(\mu+\lambda)div \mathbf{w}\gamma)\cdot[[\mathbf{v}]]ds|
\leq C\sum\limits_{i=1,2}h_{\kappa_{i}}^{-\delta}( h_{\ell}^{1-r}  \|\mathbf{f}\|_{\mathbf{L}^{2}(\kappa)}
+\mu\|\mathbf{w}\|_{\mathbf{H}^{1+r}(\kappa)}
+(\mu+\lambda)\|div\mathbf{w}\|_{r,\kappa})h_{\kappa_{i}}^{r}|\mathbf{v}|_{\mathbf{H}^{1}(\kappa_{i})},
\end{equation*}
and substituting the above inequality into \eqref{s3.12} we get
\begin{align*}
&|a_{h}(\mathbf{w},\mathbf{v})-\int_{\Omega}\mathbf{f}\cdot\mathbf{v}d\mathbf{x}|\leq  C h^{r-\delta}\|\mathbf{\nabla}_h\mathbf{v}\|_{\underset{\sim}{\mathbf{L}}^2(\Omega)}\{h^{1-r}  \|\mathbf{f}\|_{\mathbf{L}^{2}(\Omega)}+\mu\|\mathbf{w}\|_{\mathbf{H}^{1+r}(\Omega)}+(\mu+\lambda)|div\mathbf{w}|_{r,\Omega}\}
\nonumber\\
&\quad\quad\quad \leq C h^{-\delta} h^{r}\|\mathbf{f}\|_{\mathbf{L}^{2}(\Omega)}\|\mathbf{v}\|_{h},~~~\forall \mathbf{v}\in\mathbf{H}(h).
\end{align*}
Noting that $-\delta+r=-\frac{1}{2}+r+r=s$, we get the desired result. The proof is completed.
\end{proof}

Now we can state the error estimates of C-R element approximation for \eqref{s2.2}.
\begin{theorem}\label{thm3.2}
Under the conditions of Theorem 3.1, it is valid that
\begin{align}
&\|\mathbf{w}-\mathbf{w}_{h}\|_{h}\leq C h^{s}\|\mathbf{f}\|_{\mathbf{L}^{2}(\Omega)},\label{s3.19}\\
&\|\mathbf{w}-\mathbf{w}_{h}\|_{\mathbf{L}^{2}(\Omega)}\leq C h^{2s}\|\mathbf{f}\|_{\mathbf{L}^{2}(\Omega)}.\label{s3.20}
\end{align}
\end{theorem}
\begin{proof}
Combining \eqref{s2.7} and \eqref{s2.8} we deduce
\begin{equation}\label{s3.21}
\|\mathbf{w}^{*}\|_{\mathbf{H}^{1+s}(\Omega)}\leq\frac{ C}{1+\lambda}\|\mathbf{f}\|_{\mathbf{L}^{2}(\Omega)}.
\end{equation}
Referring to (5.8) in \cite{crouzeix} we have for any $\mathbf{v}\in \mathbf{H}^{1+s}(\Omega)$
\begin{equation}\label{s3.22}
(div_h\mathbf{I}_{h}\mathbf{v})|_{\kappa}=\frac{1}{|\kappa|}\int_{\kappa}div\mathbf{v} d\mathbf{x},~~~ \forall~\kappa\in \pi_{h},
\end{equation}
and
\begin{equation}\label{s3.23}
\|\mathbf{v}-\mathbf{I}_{h}\mathbf{v}\|_{\mathbf{L}^{2}(\Omega)}+h\|\mathbf{\nabla}_h(\mathbf{v}-\mathbf{I}_{h}\mathbf{v})\|_{\underset{\sim}{\mathbf{L}}^2(\Omega)}
\leq C h^{1+s}|\mathbf{v}|_{\mathbf{H}^{1+s}(\Omega)}.
\end{equation}
From \eqref{s2.6} and \eqref{s3.22} we get
\begin{equation}\label{s3.24}
div_h \mathbf{I}_{h}\mathbf{w}^{*}=\frac{1}{|\kappa|}\int_{\kappa}div\mathbf{w}^{*}d\mathbf{x}=\frac{1}{|\kappa|}\int_{\kappa}div\mathbf{w}d\mathbf{x}=div_h \mathbf{I}_{h}\mathbf{w}.
\end{equation}
By \eqref{s2.6}, \eqref{s3.24}, \eqref{s3.23} and \eqref{s3.21}, we deduce
\begin{align}
&\inf\limits_{\mathbf{v}\in\mathbf{S}_{0}^{h}(\Omega)}(\|\mathbf{w}-\mathbf{v}\|_{h})\leq \|\mathbf{w}-\mathbf{I}_{h}\mathbf{w}\|_{h} \nonumber\\
&\quad\quad\quad =(\mu\|\mathbf{\nabla}_h(\mathbf{w}-\mathbf{I}_{h}\mathbf{w})\|_{\underset{\sim}{\mathbf{L}}^2(\Omega)}^{2}
+(\mu+\lambda)\|div_h(\mathbf{w}^{*}-\mathbf{I}_{h}\mathbf{w}^{*})\|_{0,\Omega}^{2})^{\frac{1}{2}}\nonumber\\
&\quad\quad \quad\leq C h^{s}\|\mathbf{f}\|_{\mathbf{L}^{2}(\Omega)}.\label{s3.25}
\end{align}
From the Strang Lemma or (3.15) in \cite{brenner} we have
\begin{equation}\label{s3.26}
\|\mathbf{w}-\mathbf{w}_{h}\|_{h}\leq \inf\limits_{\mathbf{v}\in\mathbf{S}_{0}^{h}(\Omega)}\|\mathbf{w}-\mathbf{v}\|_{h}
+\sup\limits_{\mathbf{v}\in\mathbf{S}_{0}^{h}(\Omega)\setminus\{\mathbf{0}\}}\frac{|D_{h}(\mathbf{w},\mathbf{v})|}
{\|\mathbf{v}\|_{h}}.
\end{equation}
Substituting  \eqref{s3.25} and \eqref{s3.9} into \eqref{s3.26} we get  \eqref{s3.19}.

By Nitsche's technique, we have
\begin{align}
&\|\mathbf{w}-\mathbf{w}_{h}\|_{\mathbf{L}^{2}(\Omega)}\leq
\|\mathbf{w}-\mathbf{w}_{h}\|_{h}\sup\limits_{\mathbf{g}\in \mathbf{L}^{2}(\Omega)\setminus\{\mathbf{0}\}}\{\frac{1}{\|\mathbf{g}\|_{\mathbf{L}^{2}(\Omega)}}
\|\mathbf{\Psi}-\mathbf{\Psi}_{h}\|_{h}\} \nonumber\\
&\quad\quad \quad+
\sup\limits_{\mathbf{g}\in \mathbf{L}^{2}(\Omega)\setminus\{\mathbf{0}\}}\{\frac{1}{\|\mathbf{g}\|_{\mathbf{L}^{2}(\Omega)}}
(D_{h}(\mathbf{w}, \Psi-\Psi_{h}) +D_{h}(\Psi, \mathbf{w}-\mathbf{w}_{h}) \},\label{s3.27}
\end{align}
where for any $\mathbf{g}\in \mathbf{L}^{2}(\Omega)$, $\Psi\in \mathbf{H}_{0}^{1}(\Omega)$ is the solution of
\begin{equation}\label{s3.28}
a(\mathbf{v}, \Psi)=b(\mathbf{v},\mathbf{g}),~~~\forall \mathbf{v}\in \mathbf{H}_{0}^{1}(\Omega),
\end{equation}
and $\Psi_{h}\in \mathbf{S}_{0}^{h}(\Omega)$ is the C-R element solution of \eqref{s3.28}.\\
Using the same argument as \eqref{s3.9} and \eqref{s3.19} we get
\begin{align}
&\|\Psi-\Psi_{h}\|_{h}\leq C h^{s}\|\mathbf{g}\|_{\mathbf{L}^{2}(\Omega)},\label{s3.29}\\
&D_{h}(\Psi, \mathbf{w}-\mathbf{w}_{h})\leq C h^{s}\|\mathbf{g}\|_{L^{2}(\Omega)}\|\mathbf{w}-\mathbf{w}_{h}\|_{h}.\label{s3.30}
\end{align}
Substituting  \eqref{s3.29}, \eqref{s3.30}, \eqref{s3.9} and \eqref{s3.19} into  \eqref{s3.27} we get \eqref{s3.20}.
\end{proof}

Since \eqref{s2.4} and \eqref{s3.3} are well-posed (see \cite{brenner}), we can define two linear bounded operators
$\mathbf{T}: \mathbf{L}^{2}(\Omega)\rightarrow \mathbf{H}_{0}^{1}(\Omega)\hookrightarrow \mathbf{L}^{2}(\Omega)$ satisfying
\begin{equation}\label{s3.31}
a(\mathbf{T}\mathbf{f},\mathbf{v})=b(\mathbf{f},\mathbf{v}),~~~\forall \mathbf{v}\in \mathbf{H}_{0}^{1}(\Omega),
\end{equation}
and $\mathbf{T}_h: \mathbf{L}^{2}(\Omega)\rightarrow \mathbf{S}_{0}^{h}(\Omega)$ such that
\begin{equation}\label{s3.32}
a(\mathbf{T}_{h}\mathbf{f},\mathbf{v})=b(\mathbf{f},\mathbf{v}),~~~\forall \mathbf{v}\in \mathbf{S}_{0}^{h}(\Omega).
\end{equation}

Because of the compact inclusion $\mathbf{H}_{0}^{1}(\Omega)\hookrightarrow \mathbf{L}^{2}(\Omega)$, we know that $\mathbf{T}$ is compact.
It is easy to know that \eqref{s2.2} and \eqref{s3.2} has the following equivalent operator form, respectively:
\begin{equation*}
\mathbf{u}=\omega \mathbf{T}\mathbf{u},~~~~~~\mathbf{u}_{h}=\omega_{h} \mathbf{T}_{h}\mathbf{u}_{h}.
\end{equation*}
Thus,
\begin{equation*}
\mathbf{T}\mathbf{u}=\frac{1}{\omega}\mathbf{u},\quad\quad\mathbf{T}_{h}\mathbf{u}_{h}=\frac{1}{\omega_{h} }\mathbf{u}_{h}.
\end{equation*}
Denote $\varpi=\frac{1}{\omega}, \varpi_h=\frac{1}{\omega_{h}}$. $\varpi$ and $\varpi_h$ are called the eigenvalues of $\mathbf{T}$ and $\mathbf{T}_h$, respectively.\\
\indent From \eqref{s3.20} we have
\begin{equation*}
\|\mathbf{T}-\mathbf{T}_h\|_{\mathbf{L}^{2}(\Omega)\rightarrow \mathbf{L}^{2}(\Omega) }=\sup\limits_{\mathbf{f}\in\mathbf{L}^{2}(\Omega)\setminus\{\mathbf{0}\}}\frac{\|\mathbf{Tf}-\mathbf{T_{h}f}\|_{\mathbf{L}^{2}(\Omega)}}{\|\mathbf{f}\|_{\mathbf{L}^{2}(\Omega)}}\leq C h^{2s}\rightarrow 0~~(h\rightarrow 0).
\end{equation*}

Suppose that $\{\omega_{l}\}$ and $\{\omega_{l,h}\}$, arranged from small to large and each repeated as many
times as its multiplicity, are enumerations of the eigenvalues of
\eqref{s2.2} and \eqref{s3.2} respectively, and $\omega=\omega_{j}$ is the $j$th eigenvalue with the algebraic multiplicity $q$, $\omega=\omega_{j}=\omega_{j+1}=\cdots=\omega_{j+q-1}$.
 Since $\mathbf{T}_{h}$ converges to $\mathbf{T}$, $q$ eigenvalues $\omega_{j,h}, \omega_{j+1,h}, \cdots, \omega_{j+q-1,h}$ of \eqref{s3.2} will converge to $\omega$. Let $\mathbf{M}(\omega)$ be the space spanned by all eigenfunctions corresponding to the
eigenvalue $\omega$, and $\mathbf{M}_{h}(\omega)$ be the space spanned by all eigenfunctions of \eqref{s3.2} corresponding to the
eigenvalues $\omega_{l,h}(l=j,j+1,\cdots,j+q-1)$. Let $\widehat{\mathbf{M}}(\omega)=\{\mathbf{v}\in \mathbf{M}(\omega): \|\mathbf{v}\|_{h}=1\}$,
$\widehat{\mathbf{M}}_h(\omega)=\{\mathbf{v}\in \mathbf{M}_h(\omega): \|\mathbf{v}\|_{h}=1\}$. We also write $\mathbf{M}(\omega)=\mathbf{M}(\varpi)$,
$\mathbf{M}_h(\omega)=\mathbf{M}_h(\varpi)$, $\widehat{\mathbf{M}}(\omega)=\widehat{\mathbf{M}}(\varpi)$, and $\widehat{\mathbf{M}}_h(\omega)=\widehat{\mathbf{M}}_h(\varpi)$.

From Lemma 2.4 in \cite{yang2011} we have the following results.
\begin{theorem}\label{thm3.3}
Suppose that $\mathbf{R}(\Omega)$ holds. Let $\omega$ and $\omega_{h}$ be the $j$th eigenvalue of \eqref{s2.2} and \eqref{s3.2}, respectively, then
$\omega_{h}\rightarrow\omega$ as $h\rightarrow 0$ and
\begin{equation}\label{s3.33}
|\omega-\omega_{h}|\leq  C \|(\mathbf{T}-\mathbf{T}_h)|_{\mathbf{M}(\omega)}\|_{\mathbf{L}^{2}(\Omega)}.
\end{equation}
For any eigenfunction $\mathbf{u}_{h}$ corresponding to $\omega_{h}$, satisfying $\|\mathbf{u}_{h}\|_{h}=1$, there exists
 eigenfunction $\mathbf{u}\in \mathbf{M}(\omega)$ such that
 \begin{align}
&\|\mathbf{u}_{h}-\mathbf{u}\|_{h}\leq  \omega\|\mathbf{T}\mathbf{u}-\mathbf{T}_h\mathbf{u}\|_h+C \|(\mathbf{T}-\mathbf{T}_h)|_{\mathbf{M}(\omega)}\|_{\mathbf{L}^{2}(\Omega)},\label{s3.34}\\
&\|\mathbf{u}_{h}-\mathbf{u}\|_{\mathbf{L}^2(\Omega)}\leq  C \|(\mathbf{T}-\mathbf{T}_h)|_{\mathbf{M}(\omega)}\|_{\mathbf{L}^{2}(\Omega)}.\label{s3.35}
\end{align}
For any $\mathbf{u}\in \widehat{\mathbf{M}}(\omega)$, there exists $\mathbf{u}_h\in \mathbf{M}_h(\omega)$ such that
\begin{equation}\label{s3.36}
\|\mathbf{u}-\mathbf{u}_{h}\|_{h}\leq  C(\|(\mathbf{T}-\mathbf{T}_h)|_{\mathbf{M}(\omega)}\|_{h}+\|(\mathbf{T}-\mathbf{T}_h)|_{\mathbf{M}(\omega)}\|_{\mathbf{L}^{2}(\Omega)}) .
\end{equation}
\end{theorem}

Theorem 3.2 can also be expressed as
\begin{align*}
&\|\mathbf{T}\mathbf{f}-\mathbf{T}_{h}\mathbf{f}\|_{h}\leq C h^{s}\|\mathbf{f}\|_{\mathbf{L}^{2}(\Omega)},\\
&\|\mathbf{T}\mathbf{f}-\mathbf{T}_{h}\mathbf{f}\|_{\mathbf{L}^{2}(\Omega)}\leq C  h^{2s}\|\mathbf{f}\|_{\mathbf{L}^{2}(\Omega)},
\end{align*}
thus we have
\begin{equation}\label{s3.37}
\|(\mathbf{T}-\mathbf{T}_{h})|_{\mathbf{M}(\omega)}\|_{h}\leq  C h^{s},\quad\|(\mathbf{T}-\mathbf{T}_{h})|_{\mathbf{M}(\omega)}\|_{\mathbf{L}^2(\Omega)}\leq  C h^{2s}.
\end{equation}

\section{Two-grid discretizations for the elastic eigenvalue problem}\label{frame-model}

In this section, we will establish two-grid discretization schemes for the elastic eigenvalue problem.

\indent Let $\pi_{H}(\Omega)$ be a regular triangulation of size $H\in (0,1)$ and $\pi_{h}(\Omega)$ ($h\ll H$) be a fine grid refined from $\pi_{H}(\Omega)$.

\noindent{\bf Scheme 4.1.} Two-grid discretization based on inverse iteration

\indent{\bf Step 1.} Solve \eqref{s3.2} on a coarse grid $\pi_{H}(\Omega)$: Find $\omega_{H}\in \mathbb{R}$, $\mathbf{u}_{H}\in \mathbf{S}_{0}^{H}(\Omega)$ such that $\|\mathbf{u}_{H}\|_{H}=1$ and
\begin{equation*}
a_{H}(\mathbf{u}_{H},\mathbf{v})=\omega_{H}b(\mathbf{u}_{H},\mathbf{v}),~~~\forall \mathbf{v}\in \mathbf{S}_{0}^{H}(\Omega).
\end{equation*}
\indent{\bf Step 2.} Solve a linear boundary value problem on a fine grid  $\pi_{h}(\Omega)$:
find $\mathbf{u}^{h}\in \mathbf{S}_{0}^{h}(\Omega)$ such that
\begin{equation*}%\label{s4.1}
a_h(\mathbf{u}^{h} ,\mathbf{v})=\omega_{H}b(\mathbf{u}_{H},\mathbf{v}),~~~\forall \mathbf{v}\in \mathbf{S}_{0}^{h}(\Omega).
\end{equation*}
\indent{\bf Step 3.} Compute the Rayleigh quotient
\begin{equation*}
\omega^{h}=\frac{a_h(\mathbf{u}^{h},\mathbf{u}^{h})}{b(\mathbf{u}^{h},\mathbf{u}^{h})}.
\end{equation*}

\noindent{\bf Scheme 4.2.} Two-grid discretization based on the shifted-inverse iteration

\indent{\bf Step 1.} Solve \eqref{s3.2} on a coarse grid $\pi_{H}(\Omega)$: Find $\omega_{H}\in \mathbb{R}$, $\mathbf{u}_{H}\in \mathbf{S}_{0}^{H}(\Omega)$ such that $\|\mathbf{u}_{H}\|_{H}=1$ and
\begin{equation*}
a_{H}(\mathbf{u}_{H},\mathbf{v})=\omega_{H}b(\mathbf{u}_{H},\mathbf{v}),~~~\forall \mathbf{v}\in \mathbf{S}_{0}^{H}(\Omega).
\end{equation*}
\indent{\bf Step 2.} Solve a linear boundary value problem on a fine grid  $\pi_{h}(\Omega)$:
find $\mathbf{u}'\in \mathbf{S}_{0}^{h}(\Omega)$ such that
\begin{equation*}%\label{s4.2}
a_h(\mathbf{u}' ,\mathbf{v})-\omega_{H}b(\mathbf{u}',\mathbf{v})=b(\mathbf{u}_{H},\mathbf{v}),~~~\forall \mathbf{v}\in \mathbf{S}_{0}^{h}(\Omega),
\end{equation*}
and set $\mathbf{u}^{h}=\mathbf{u}'/\|\mathbf{u}'\|_{h}$.\\
\indent{\bf Step 3.} Compute the Rayleigh quotient
\begin{equation*}
\omega^{h}=\frac{a_h(\mathbf{u}^{h},\mathbf{u}^{h})}{b(\mathbf{u}^{h},\mathbf{u}^{h})}.
\end{equation*}

\begin{lemma}\label{lem4.1}
Let $(\omega, \mathbf{u})$ be an eigenpair of \eqref{s2.2}, then, for any $\mathbf{v}\in \mathbf{H}(h)$ with $\|\mathbf{v}\|_{ \mathbf{L}^{2}(\Omega)}\neq 0$, the generalized Rayleigh quotient satisfies
\begin{equation*}%\label{s4.3}
\frac{a_{h}(\mathbf{v}, \mathbf{v})}{\|\mathbf{v}\|_{ \mathbf{L}^{2}(\Omega)}^2}-\omega=\frac{a_{h}(\mathbf{u}-\mathbf{v}, \mathbf{u}-\mathbf{v})}{\|\mathbf{v}\|_{ \mathbf{L}^{2}(\Omega)}^2}-\omega\frac{\|\mathbf{u}-\mathbf{v}\|_{\mathbf{L}^{2}(\Omega)}^2}{\|\mathbf{v}\|_{ \mathbf{L}^{2}(\Omega)}^2}+2\frac{D_{h}(\mathbf{u}, \mathbf{v})}{\|\mathbf{v}\|_{ \mathbf{L}^{2}(\Omega)}^2}.
\end{equation*}
 \end{lemma}
\begin{proof}
 For any $\mathbf{v}\in \mathbf{H}(h)$, from \eqref{s2.2}, \eqref{s3.31} and \eqref{s3.4} we have
\begin{equation*}
D_{h}(\mathbf{u},\mathbf{v})=a_{h}(\mathbf{u},\mathbf{v})- b(\omega\mathbf{u}, \mathbf{v}),
\end{equation*}
thus,
\begin{align*}
&a_{h}(\mathbf{u}-\mathbf{v}, \mathbf{u}-\mathbf{v})-\omega b(\mathbf{u}-\mathbf{v}, \mathbf{u}-\mathbf{v})\\
&\quad\quad =a_{h}(\mathbf{u}, \mathbf{u})+a_{h}(\mathbf{v}, \mathbf{v})-2a_{h}(\mathbf{u}, \mathbf{v})-\omega(b(\mathbf{u}, \mathbf{u})+b(\mathbf{v}, \mathbf{v})-2b(\mathbf{u}, \mathbf{v})) \nonumber\\
&\quad\quad =\omega b(\mathbf{u}, \mathbf{u})+a_{h}(\mathbf{v}, \mathbf{v})-2D_h(\mathbf{u}, \mathbf{v})-\omega b(\mathbf{u}, \mathbf{u})-\omega b(\mathbf{v}, \mathbf{v})\\
&\quad\quad =a_{h}(\mathbf{v}, \mathbf{v})-\omega b(\mathbf{v}, \mathbf{v})-2D_h(\mathbf{u}, \mathbf{v}),
\end{align*}
and dividing $\|\mathbf{v}\|_{ \mathbf{L}^{2}(\Omega)}^2$ in both sides of the above we
obtain the desired conclusion.
\end{proof}

\begin{theorem}\label{thm4.1}
Suppose that $\mathbf{R}(\Omega)$ holds.
 Assume that $(\omega^{h}, \mathbf{u}^{h})$ is an approximate eigenpair obtained
by Scheme 4.1. Then there exists an eigenfunction $\mathbf{u}\in \mathbf{M}(\omega)$ such that
\begin{align}
&\|\mathbf{u}^h-\mathbf{u}\|_{h}\leq C(H^{2s}+h^{s}),\label{s4.1}\\
&|\omega^{h}-\omega|\leq C(H^{4s}+h^{2s}).\label{s4.2}
\end{align}
\end{theorem}
\begin{proof}
Let $\mathbf{u}\in \mathbf{M}(\omega)$ such that $\mathbf{u}_{H}-\mathbf{u}$ and $\omega_{H}-\omega$ satisfy Theorem 3.3.
Since $\mathbf{u}=\omega \mathbf{T}\mathbf{u}$, and from the definition of $\mathbf{T}_{h}$ and Step 2 in Scheme 4.1 we get $\mathbf{u}^{h}=\omega_{H} \mathbf{T}_{h}\mathbf{u}_{H}$, then, from Theorem 3.3, noting  \eqref{s3.37}, we deduce
\begin{align*}
&\|\mathbf{u}^{h}-\mathbf{u}\|_{h}= \|\omega_{H}\mathbf{T}_{h}\mathbf{u}_{H}-\omega \mathbf{T}\mathbf{u}\|_{h}\\
&\quad\quad =\|\omega_{H}(\mathbf{T}_{h}\mathbf{u}_{H}-\mathbf{T}_{h}\mathbf{u})+\omega_{H}(\mathbf{T}_{h}\mathbf{u}- \mathbf{T}\mathbf{u})+(\omega_H-\omega)\mathbf{T}\mathbf{u}\|_{h}\\
&\quad\quad \leq |\omega_H|\cdot \|\mathbf{T}_{h}(\mathbf{u}_{H}-\mathbf{u})\|_{h}+|\omega_H|\cdot \|\mathbf{T}_{h}\mathbf{u}-\mathbf{T}\mathbf{u}\|_{h}+|\omega_H-\omega|\cdot \|\mathbf{T}\mathbf{u}\|_{h}\\
&\quad\quad \leq C(\|\mathbf{T}_{h}(\mathbf{u}_{H}-\mathbf{u})-\mathbf{T}(\mathbf{u}_{H}-\mathbf{u})\|_{h}+\|\mathbf{T}(\mathbf{u}_{H}- \mathbf{u})\|_{h})+Ch^s+CH^{2s}\\
&\quad\quad \leq C( h^{s}\|\mathbf{u}_{H}-\mathbf{u}\|_{\mathbf{L}^{2}(\Omega)}+\|\mathbf{u}_{H}-\mathbf{u})\|_{\mathbf{L}^{2}(\Omega)})+Ch^s+CH^{2s}\\
&\quad\quad \leq C(H^{2s}+h^s),
\end{align*}
namely,  \eqref{s4.1} is valid.

Because $\mathbf{u}\in \mathbf{H}^1_0(\Omega)$ and $\mathbf{u}^{h}\in \mathbf{S}_0^h(\Omega)$ are piecewise $\mathbf{H}^1$-functions,
using the Poincar$\acute{e}$-Friedrichs inequality (cf. (1.5) in \cite{brenner2003}) we have
\begin{equation}\label{s4.3}
\|\mathbf{u}^{h}-\mathbf{u}\|_{\mathbf{L}^2(\Omega)}\leq C|\mathbf{u}^{h}-\mathbf{u}|_{1,h}\leq C\|\mathbf{u}^{h}-\mathbf{u}\|_{h},
\end{equation}
thus, combing with \eqref{s4.1} we get
\begin{equation}\label{s4.4}
\|\mathbf{u}^{h}-\mathbf{u}\|_{\mathbf{L}^{2}(\Omega)}\leq C\|\mathbf{u}^{h}-\mathbf{u}\|_{h}\leq C(H^{2s}+h^{s}).
\end{equation}
Since $D_h(\mathbf{u}, \mathbf{u})=0$, from Theorem 3.1 and  \eqref{s4.1} we have
 \begin{align}
&|D_h(\mathbf{u}, \mathbf{u}^{h})|=|D_h(\mathbf{u}, \mathbf{u}^{h}-\mathbf{u})|\leq  C h^s\|\mathbf{u}\|_{ \mathbf{L}^{2}(\Omega)}\|\mathbf{u}^{h}-\mathbf{u}\|_{h} \nonumber\\
&\quad\quad \leq C h^s(H^{2s}+h^s)\|\mathbf{u}\|_{ \mathbf{L}^{2}(\Omega)}\nonumber\\
&\quad\quad \leq C(h^{2s}+h^{s}H^{2s}).\label{s4.5}
\end{align}

From Lemma 4.1 we have
\begin{align}
&\omega^{h}-\omega=\frac{a_{h}(\mathbf{u}^{h},\mathbf{u}^{h})}{b(\mathbf{u}^{h},\mathbf{u}^{h})}-\omega \nonumber\\
&\quad\quad =\frac{a_{h}(\mathbf{u}-\mathbf{u}^{h}, \mathbf{u}-\mathbf{u}^{h})}{\|\mathbf{u}^{h}\|_{ \mathbf{L}^{2}(\Omega)}^2}-\omega\frac{\|\mathbf{u}-\mathbf{u}^{h}\|_{\mathbf{L}^{2}(\Omega)}^2}{\|\mathbf{u}^{h}\|_{ \mathbf{L}^{2}(\Omega)}^2}+2\frac{D_{h}(\mathbf{u}, \mathbf{u}^{h})}{\|\mathbf{u}^{h}\|_{ \mathbf{L}^{2}(\Omega)}^2}\nonumber\\
&\quad\quad =\frac{\|\mathbf{u}-\mathbf{u}^{h}\|_{h}^2}{\|\mathbf{u}^{h}\|_{ \mathbf{L}^{2}(\Omega)}^2}-\omega\frac{\|\mathbf{u}-\mathbf{u}^{h}\|_{\mathbf{L}^{2}(\Omega)}^2}{\|\mathbf{u}^{h}\|_{ \mathbf{L}^{2}(\Omega)}^2}+2\frac{D_{h}(\mathbf{u}, \mathbf{u}^{h})}{\|\mathbf{u}^{h}\|_{ \mathbf{L}^{2}(\Omega)}^2}.\label{s4.6}
\end{align}
Substituting \eqref{s4.1}, \eqref{s4.4} and \eqref{s4.5} into  \eqref{s4.6} we get \eqref{s4.2}.
 The proof is completed.
\end{proof}

Let $(\omega_{j}, \mathbf{u}_{j})$ and $(\omega_{j,h},
\mathbf{u}_{j,h})$ be the $j$th eigenpair of \eqref{s2.2} and \eqref{s3.2}, respectively.
Denote $dist(\mathbf{u},S)=\inf\limits_{\mathbf{v}\in S}\|\mathbf{u}-\mathbf{v}\|_{h}$.

The following lemma is an analog to Theorem 3.2 in \cite{yang2011} and  Lemma 4.1 in \cite{yangbihanyu}, and can be proved similarly.
\begin{lemma}\label{lem4.2}
Let $(\varpi_{0}, \mathbf{u}_{0})$ be an approximation for
$(\varpi_j, \mathbf{u}_{j})$
where $\varpi_{0}$ is not an eigenvalue of $\mathbf{T}_{h}$ and
$\mathbf{u}_{0}\in \mathbf{S}_0^{h}(\Omega)$ with $\|\mathbf{u}_{0}\|_h=1$.
 Suppose that

(C1)~$dist(\mathbf{u}_{0}, \mathbf{M}_{h}(\varpi_{j}))\leq \frac{1}{2}$;

(C2)~$|\varpi_{0}-\varpi_{j}|\leq \frac{\varrho}{4}$,
$|\varpi_{k,h}-\varpi_{k}|\leq \frac{\varrho}{4}$ for
$k=j-1,j,\cdots,j+q~(k\not=0)$;

(C3)~ $\mathbf{u}'\in \mathbf{S}_0^{h}(\Omega), \mathbf{u}^{h}\in \mathbf{S}_0^{h}(\Omega)$ satisfy
\begin{equation*}
(\varpi_{0}-\mathbf{T}_{h})\mathbf{u}'=\mathbf{u}_{0},~~~\mathbf{u}^{h}=\frac{\mathbf{u}'}{\|\mathbf{u}'\|_h}.
\end{equation*}
Then
\begin{equation}\label{s4.7}
dist(\mathbf{u}^{h},\widehat{\mathbf{M}}_{h}(\varpi
_{j}))\leq
\frac{4}{\varrho}\max\limits_{j\leq k\leq j+q-1}|\varpi_{0}-\varpi_{k,h}|dist(\mathbf{u}_{0}, \mathbf{M}_{h}(\varpi_{j})),
\end{equation}
 where
$\varrho=\min\limits_{\varpi_{k}\not=\varpi_{j}}|\varpi_{k}-\varpi_{j}|$ is
the separation constant of the eigenvalue $\varpi_{j}$.
 \end{lemma}

\begin{theorem}\label{thm4.2}
Suppose that $\mathbf{R}(\Omega)$ holds.
Assume that $(\omega^{h}, \mathbf{u}^{h})$ is an approximate eigenpair obtained
by Scheme 4.2. Then there exists an eigenfunction $\mathbf{u}\in \mathbf{M}(\omega)$ such that
\begin{align}
&\|\mathbf{u}^h-\mathbf{u}_j\|_{h}\leq C(H^{4s}+h^{s}),\label{s4.8}\\
&|\omega^{h}-\omega_j|\leq C(H^{8s}+h^{2s}).\label{s4.9}
\end{align}
\end{theorem}
\begin{proof}
We use Lemma 4.2 to complete the proof. Select
\begin{equation*}
\varpi_{0}=\frac{1}{\omega_{H}}~~ and~~
\mathbf{u}_{0}=\frac{\omega_{H}\mathbf{T}_{h}\mathbf{u}_{H}}{\|\omega_{H}\mathbf{T}_{h}\mathbf{u}_{H}\|_{h}}.
\end{equation*}
From Theorem 3.3 we know that there exists $\widetilde{\mathbf{u}}\in \mathbf{M}(\omega_j)$ making $\mathbf{u}_{H}-\widetilde{\mathbf{u}}$ satisfy \eqref{s3.34} and \eqref{s3.35}.

From \eqref{s3.32}, Schwarz inequality, and \eqref{s3.35} we deduce
\begin{align*}
&a_{h}(\mathbf{T}_{h}(\mathbf{u}_{H}-\widetilde{\mathbf{u}}),\mathbf{T}_{h}(\mathbf{u}_{H}-\widetilde{\mathbf{u}}) )=b(\mathbf{u}_{H}-\widetilde{\mathbf{u}},\mathbf{T}_{h}(\mathbf{u}_{H}-\widetilde{\mathbf{u}}))\\
&\quad \quad \leq\|\mathbf{u}_{H}-\widetilde{\mathbf{u}}\|_{\mathbf{L}^{2}(\Omega)}\|\mathbf{T}_{h}(\mathbf{u}_{H}-\widetilde{\mathbf{u}})\|_{\mathbf{L}^{2}(\Omega)}\leq
C \|(\mathbf{T}-\mathbf{T}_H)|_{\mathbf{M}(\omega_j)}\|_{\mathbf{L}^{2}(\Omega)}^{2},
\end{align*}
thus,
\begin{equation*}%\label{s4.12}
\|\mathbf{T}_{h}(\mathbf{u}_{H}-\widetilde{\mathbf{u}})\|_{h}\leq
C \|(\mathbf{T}-\mathbf{T}_H)|_{\mathbf{M}(\omega_j)}\|_{\mathbf{L}^{2}(\Omega)},
\end{equation*}
then, combining with \eqref{s3.33} and $\|\mathbf{T}_{h}\widetilde{\mathbf{u}}-\mathbf{T}\widetilde{\mathbf{u}}\|_{h}\leq C\|(\mathbf{T}_{h}-\mathbf{T})|_{\mathbf{M}(\omega_j)}\|_{h}$, we derive
\begin{align*}
&\|\omega_{H}\mathbf{T}_{h}\mathbf{u}_{H}-\widetilde{\mathbf{u}}\|_{h}=\|\omega_{H}\mathbf{T}_{h}\mathbf{u}_{H}-\omega_j\mathbf{T}\widetilde{\mathbf{u}}\|_{h}\\
&\quad\quad  =\|\omega_{H}(\mathbf{T}_{h}\mathbf{u}_{H}-\mathbf{T}_{h}\widetilde{\mathbf{u}})+\omega_{H}(\mathbf{T}_{h}\widetilde{\mathbf{u}}-\mathbf{T}\widetilde{\mathbf{u}})+(\omega_{H}-\omega_j)\mathbf{T}\widetilde{\mathbf{u}}\|_{h}\nonumber\\
&\quad \quad \leq C(\|(\mathbf{T}-\mathbf{T}_H)|_{\mathbf{M}(\omega_j)}\|_{\mathbf{L}^{2}(\Omega)}+\|(\mathbf{T}-\mathbf{T}_h)|_{\mathbf{M}(\omega_j)}\|_{h}).
\end{align*}
It is easy to prove that in any normed space, it is valid for any nonzero $\Phi,\Psi$ that
\begin{equation*}
\|\frac{\Phi}{\|\Phi\|}-\frac{\Psi}{\|\Psi\|}\|\leq 2\frac{\|\Phi-\Psi\|}{\|\Phi\|},~~~
\|\frac{\Phi}{\|\Phi\|}-\frac{\Psi}{\|\Psi\|}\|\leq 2\frac{\|\Phi-\Psi\|}{\|\Psi\|}.
\end{equation*}
Hence,
\begin{align}
\|\mathbf{u}_{0}-\frac{\widetilde{\mathbf{u}}}{\|\widetilde{\mathbf{u}}\|_{h}}\|_{h}&=\|\frac{\omega_{H}\mathbf{T}_{h}\mathbf{u}_{H}}{\|\omega_{H}\mathbf{T}_{h}\mathbf{u}_{H}\|_{h}}-\frac{\widetilde{\mathbf{u}}}{\|\widetilde{\mathbf{u}}\|_{h}}\|_{h}\leq C\|\omega_{H}\mathbf{T}_{h}\mathbf{u}_{H}-\widetilde{\mathbf{u}}\|_{h}\nonumber\\
&\leq C(\|(\mathbf{T}-\mathbf{T}_H)|_{\mathbf{M}(\omega_j)}\|_{\mathbf{L}^{2}(\Omega)}+\|(\mathbf{T}-\mathbf{T}_h)|_{\mathbf{M}(\omega_j)}\|_{h}).\label{s4.10}
\end{align}
For $\frac{\widetilde{\mathbf{u}}}{\|\widetilde{\mathbf{u}}\|_{h}}\in \widehat{\mathbf{M}}(\omega_j)$, from \eqref{s3.36} we know there exists $\mathbf{u}_{h}\in \mathbf{M}_{h}(\omega_j)$ such that
\begin{equation}\label{s4.11}
\|\frac{\widetilde{\mathbf{u}}}{\|\widetilde{\mathbf{u}}\|_{h}}-\mathbf{u}_{h}\|_{h}\leq
C(\|(\mathbf{T}-\mathbf{T}_{h})|_{\mathbf{M}(\omega_j)}\|_{h}+\|(\mathbf{T}-\mathbf{T}_{h})|_{\mathbf{M}(\omega_j)}\|_{\mathbf{L}^{2}(\Omega)}).
\end{equation}
From the triangle inequality, \eqref{s4.10} and \eqref{s4.11}, we have
\begin{align}
dist(\mathbf{u}_{0},\mathbf{M}_{h}(\omega_j))&\leq \|\mathbf{u}_{0}-\mathbf{u}_{h}\|_{h}\leq
\|\mathbf{u}_{0}-\frac{\widetilde{\mathbf{u}}}{\|\widetilde{\mathbf{u}}\|_{h}}\|_{h}+\|\mathbf{u}_{h}-\frac{\widetilde{\mathbf{u}}}{\|\widetilde{\mathbf{u}}\|_{h}}\|_{h}\nonumber\\ &\leq C(\|(\mathbf{T}-\mathbf{T}_H)|_{\mathbf{M}(\omega_j)}\|_{\mathbf{L}^{2}(\Omega)}+\|(\mathbf{T}-\mathbf{T}_h)|_{\mathbf{M}(\omega_j)}\|_{h}),\label{s4.12} \end{align}
then Condition (C1) in Lemma 4.2 holds when $H$ and $h$ are small enough.

From \eqref{s3.33} we know that Condition (C2) in Lemma 4.2 holds.

From Step 2 in Scheme 4.2, we know that $\mathbf{u}^{h}$ satisfies
\begin{equation*}
(\frac{1}{\omega_{H}}-\mathbf{T}_{h})\mathbf{u}'=\mathbf{u}_{0},~~~\mathbf{u}^{h}=\frac{\mathbf{u}'}{\|\mathbf{u}'\|_{h}},
\end{equation*}
that is, Condition (C3) in Lemma 4.2 holds.

Let the eigenfunctions $\{\mathbf{u}_{l,h}\}_{l=j}^{j+q-1}$ be a normalized orthonormal
basis of $\mathbf{M}_{h}(\omega_j)$ in the sense of norm $\|\cdot\|_{h}$,
then by Theorem 3.3 we know that there exist
$\{\mathbf{u}_{l}^{0}\}_{l=j}^{j+q-1}\subset \mathbf{M}(\omega_j)$ making
\eqref{s3.34} hold. Let
\begin{equation*}
\mathbf{u}^{*}=\sum\limits_{l=j}^{j+q-1}a_{h}(\mathbf{u}^{h}, \mathbf{u}_{l,h})\mathbf{u}_{l,h},
\end{equation*}
then, by \eqref{s4.7} and \eqref{s4.12} we deduce
\begin{align}
&\|\mathbf{u}^{h}-\mathbf{u}^{*}\|_{h}=dist(\mathbf{u}^{h},\mathbf{M}_{h}(\omega_j))\leq dist(\mathbf{u}^{h},\widehat{\mathbf{M}}_{h}(\omega_j))\nonumber\\
&\quad \quad\leq C\max\limits_{j\leq k\leq j+q-1}|\varpi_{0}-\varpi_{k,h}|(\|(\mathbf{T}-\mathbf{T}_{H})|_{\mathbf{M}(\omega_j)}\|_{\mathbf{L}^{2}(\Omega)}^{2}+\|(\mathbf{T}-\mathbf{T}_{h})|_{\mathbf{M}(\omega_j)}\|_{h}).\label{s4.13}
\end{align}
From \eqref{s3.33} we get
\begin{equation*}%\label{s4.16}
|\varpi_{0}-\varpi_{k,h}|=|\frac{\omega_{k,h}-\omega_{j}+\omega_{j}-\omega_{H}}{\omega_{H}\omega_{k,h}}|
\leq C\|(\mathbf{T}-\mathbf{T}_{H})|_{\mathbf{M}(\omega_j)}\|_{\mathbf{L}^{2}(\Omega)},
\end{equation*}
which together with \eqref{s4.13} yields
\begin{equation}
\|\mathbf{u}^{h}-\mathbf{u}^{*}\|_{h}\leq C(\|(\mathbf{T}-\mathbf{T}_{H})|_{\mathbf{M}(\omega_j)}\|_{\mathbf{L}^{2}(\Omega)}^{2} +\|(\mathbf{T}-\mathbf{T}_{H})|_{\mathbf{M}(\omega_j)}\|_{\mathbf{L}^{2}(\Omega)}\|(\mathbf{T}-\mathbf{T}_{h})|_{\mathbf{M}(\omega_j)}\|_{h}).\label{s4.14}
\end{equation}
Let $\mathbf{u}=\sum\limits_{l=j}^{j+q-1}a_{h}(\mathbf{u}^{h},\mathbf{u}_{l,h})\mathbf{u}_{l}^{0}$,
then, from \eqref{s3.34} we get
\begin{align}
\|\mathbf{u}^{*}-\mathbf{u}\|_{h}
&=\|\sum\limits_{l=j}^{j+q-1}a_{h}(\mathbf{u}^{h},\mathbf{u}_{l,h})(\mathbf{u}_{l,h}-\mathbf{u}_{l}^{0})\|_{h}\nonumber\\
&\leq C(\|(\mathbf{T}-\mathbf{T}_{h})|_{\mathbf{M}(\omega_j)}\|_{h}+\|(\mathbf{T}-\mathbf{T}_{h})|_{\mathbf{M}(\omega_j)}\|_{\mathbf{L}^{2}(\Omega)}).\label{s4.15}
\end{align}
From the triangle inequality, \eqref{s4.14}, \eqref{s4.15} and \eqref{s3.37} we obtain \eqref{s4.8}.

Similar to the proof of \eqref{s4.2}, from \eqref{s4.8}, \eqref{s4.3} and Lemma 4.1 we get \eqref{s4.9}.
\end{proof}

\section{Numerical experiments}\label{comparison}

In this section, we will report some numerical experiments
to verify our theoretical analysis and the efficiency of two-grid schemes. We use MATLAB 2012a to compute on a DELL inspiron5480 PC with 8G memory. Our program is implemented using the package iFEM \cite{chen}. The symbol $'-'$ in our tables means that the calculation cannot proceed since the computer runs out of memory.

{\bf Example 5.1.} Consider the elastic eigenvalue problem \eqref{s2.2} in the unit square $\Omega_S=[0,1]\times[0,1]$ and the L-shaped domain
$\Omega_L=[0,1]\times[0,1]\setminus[\frac{1}{2},1]\times[\frac{1}{2},1]$
with the density $\rho\equiv 1$. We compute the first numerical eigenvalue of \eqref{s2.2} in $\Omega_S$ and $\Omega_L$ by the nonconforming C-R element on uniformly refined meshes, and the results are denoted by $\omega_{h}^S$ and $\omega_{h}^L$, respectively. The numerical results are listed in Tables 5.1-5.2. Since the exact eigenvalues are unknown, we use the following formula
\begin{equation*}
ratio(\omega_{h})\approx lg|\frac{\omega_{h}-\omega_{h/2}}{\omega_{h/2}-\omega_{h/4}}|/lg2
\end{equation*}
to compute the approximate convergence order.

\begin{table}
\caption{The first numerical eigenvalue in $\Omega_S$ and $\Omega_L$ by direct computation with $\mu=\lambda=1$.}
\begin{center}
\begin{tabular}{ccccc}\hline
$h$&$\omega_{h}^S$& $ratio(\omega_{h}^S)$  &  $\omega_{h}^L$& $ratio(\omega_{h}^L)$ \\
\hline
$\frac{\sqrt{2}}{16}$ &  36.968038  & 1.9334 & 53.318789 &   1.3219 \\
$\frac{\sqrt{2}}{32}$ &  37.188573  & 1.9666 & 53.940006 &  1.2829\\
$\frac{\sqrt{2}}{64}$ &  37.246310  & 1.9833 & 54.188497 &  1.2410 \\
$\frac{\sqrt{2}}{128}$&  37.261082  & 1.9908 & 54.291332 &  1.2225\\
$\frac{\sqrt{2}}{256}$&  37.264818  &        & 54.334839 &         \\
$\frac{\sqrt{2}}{512}$&  37.265758  &        & 54.353484 &    \\\hline
\end{tabular}
\end{center}
\end{table}
\begin{table}
\caption{The first numerical eigenvalue in $\Omega_S$ and $\Omega_L$ by direct computation with $\mu=1, \lambda=50$.}
\begin{center}
\begin{tabular}{ccccc}\hline
$h$&$\omega_{h}^S$& $ratio(\omega_{h}^S)$  &  $\omega_{h}^L$& $ratio(\omega_{h}^L)$ \\
\hline
$\frac{\sqrt{2}}{16}$ & 51.823020 & 1.9472 &  118.462292  & 1.3352\\
$\frac{\sqrt{2}}{32}$ & 52.164464 & 1.9845 &  123.764088  & 1.2887 \\
$\frac{\sqrt{2}}{64}$ & 52.253005 & 1.9958 &  125.865404  & 1.2300 \\
$\frac{\sqrt{2}}{128}$& 52.275380 & 1.9995 &  126.725522  & 1.1814 \\
$\frac{\sqrt{2}}{256}$& 52.280990 &        &  127.092195  &  \\
$\frac{\sqrt{2}}{512}$& 52.282393 &        &  127.253875  &   \\\hline
\end{tabular}
\end{center}
\end{table}

 From Tables 5.1-5.2 we can see that the numerical eigenvalues are convergent at different values of $\lambda$, and the convergence order of the first eigenvalue $\omega_h$ is approximately equal to $2.00$, i.e., $2s\approx 2.00$ or $s\approx 1.00$ in the square.
Unfortunately, because of the computer memory limitation  we cannot continue to compute to make the convergence order stable in the L-shaped domain. According to the current results, the convergence order is approximately equal to $1.20$, i.e., $2s\approx 1.20$ or $s\approx 0.60$.

{\bf Example 5.2.} Consider the elastic eigenvalue problem \eqref{s2.2} in the L-shaped domain $\Omega_L=
[0,1]\times[0,1]\setminus[\frac{1}{2},1]\times[\frac{1}{2},1]$ with density $\rho \equiv 1$.
We compute the first approximate eigenvalue of this problem by Schemes 4.1 and 4.2, and denote the numerical eigenvalues obtained by Schemes 4.1 and 4.2 by $\omega_{(1)}^{h}$ and $\omega_{(2)}^{h}$, respectively. The numerical results are listed in Tables 5.3-5.4.
For comparison, we also solve this problem on fine grid directly by using Matlab command $eigs(A,M,1,'sm')$, and the results are denoted by $\omega_h$.

\begin{table}
\caption{The first numerical eigenvalue in the L-shaped domain obtained by Schemes 4.1 and 4.2 with $\mu=\lambda=1$.}
\begin{center}
\begin{tabular}{cccccccc}\hline
$H$&$h$&$\omega_{(1)}^{h}$&$times(s)$&$\omega_{(2)}^{h}$&$times(s)$&$\omega_h$&$times(s)$\\\hline
$\frac{\sqrt{2}}{8}$&$\frac{\sqrt{2}}{64}$   &54.375337 &0.31   &54.189403 &0.18    &54.188497 & 0.38   \\
$\frac{\sqrt{2}}{16}$&$\frac{\sqrt{2}}{256}$ &54.355380 &3.04   &54.334860 &3.48    &54.334839 & 9.03  \\
$\frac{\sqrt{2}}{32}$&$\frac{\sqrt{2}}{512}$ &54.357151 &14.01  &54.353485 &18.02   &54.353484 & 43.73 \\
$\frac{\sqrt{2}}{32}$&$\frac{\sqrt{2}}{1024}$&54.364492 &67.43  &54.361530 &106.46&   --     & -- \\\hline
\end{tabular}
\end{center}
\footnotesize time(s): the CPU time(s) from the program starting to the current calculating result appearing.
\end{table}
\begin{table}
\caption{The first numerical eigenvalue in the L-shaped domain with $\mu=1, \lambda=50$.}
\begin{center}
\begin{tabular}{cccccccc}\hline
$H$&$h$&$\omega_{(1)}^{h}$&$times(s)$&$\omega_{(2)}^{h}$&$times(s)$&$\omega_h$&$times(s)$\\\hline
$\frac{\sqrt{2}}{8}$&$\frac{\sqrt{2}}{64}$  & 127.306606  & 0.25  & 125.953352 & 0.17   & 125.865404 & 0.39    \\
$\frac{\sqrt{2}}{16}$&$\frac{\sqrt{2}}{256}$& 127.299866  & 3.10  & 127.095813 & 3.43   & 127.092195 & 9.04 \\
$\frac{\sqrt{2}}{32}$&$\frac{\sqrt{2}}{512}$& 127.295589  & 13.73 & 127.254021 & 17.85  & 127.253875 & 43.70 \\
$\frac{\sqrt{2}}{32}$&$\frac{\sqrt{2}}{1024}$&127.366437  & 66.25 & 127.327009 & 111.55 &--        & -- \\\hline
\end{tabular}
\end{center}
\end{table}
The results in Tables 5.3-5.4 show that we can use less time by two-grid discretization schemes to get the same accurate approximations as those obtained by direct computation.

 {\bf Remark.} In Tables 5.3-5.4, for the sake of list, we make the diameters of coarse grid and fine grid satisfying $H=O(\sqrt{h})$. For Scheme 4.2 we can choose $H=O(\sqrt[4]{h})$ according to Theorem 4.2. In the case of $\mu=1, \lambda=1$, when we select $H=\frac{\sqrt{2}}{8}, h=\frac{\sqrt{2}}{1024}$, it takes 119.78s to get the calculating result $\omega^{h^{(2)}}= 54.362491$.

 To observe the influence of the Lam$\acute{e}$ parameter $\lambda$, we also depict the error curves of approximations for the first eigenvalue of  \eqref{s2.2}. Since the exact eigenvalue is not known, we plot the ``error"  $|\omega_{h}-\omega_{h/2}|$, $|\omega_{(1)}^{h}-\omega_{(1)}^{h/2}|$ and $|\omega_{(2)}^{h}-\omega_{(2)}^{h/2}|$  where $h=\frac{\sqrt{2}}{256}$ by taking
$\lambda=1,10^3, 10^5, 10^8$ while $\mu$ is fixed at 1. From Fig. 5.1 we can see that in the L-shaped domain, the ``error" curves become stable as $\lambda$ increases, which indicates that the nonconforming C-R element method and the two-grid schemes of C-R element are locking-free. In the unit square, the error curves of two-grid schemes keep stable while that of direct computation jumps at $\lambda=10^8$, which leaves us a question. Frustratingly, we cannot do more sophisticated calculations at present.

\begin{figure}[h!]%[h]fix picture position
  \centering
    \includegraphics[width=2.9in]{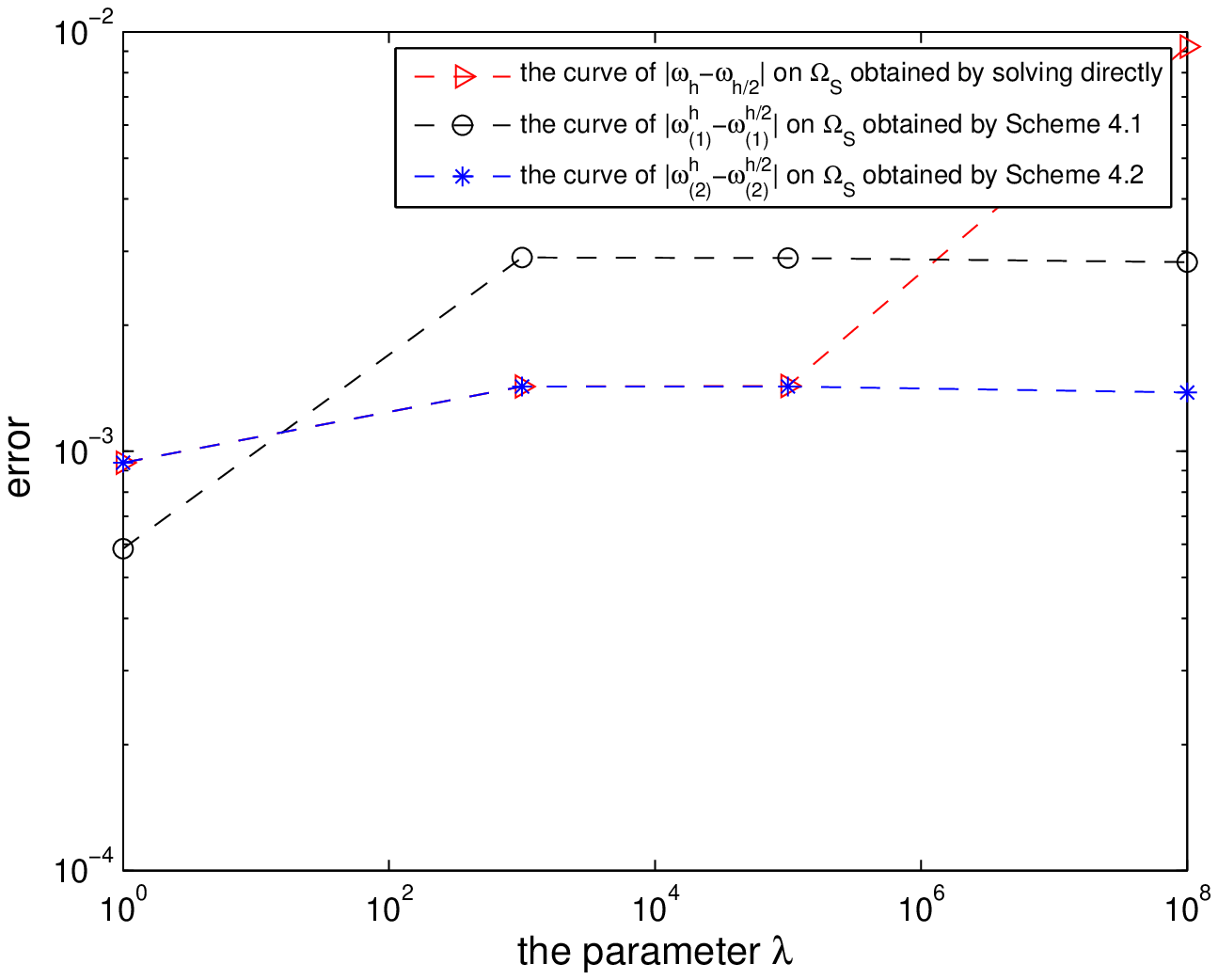}
    \includegraphics[width=2.9in]{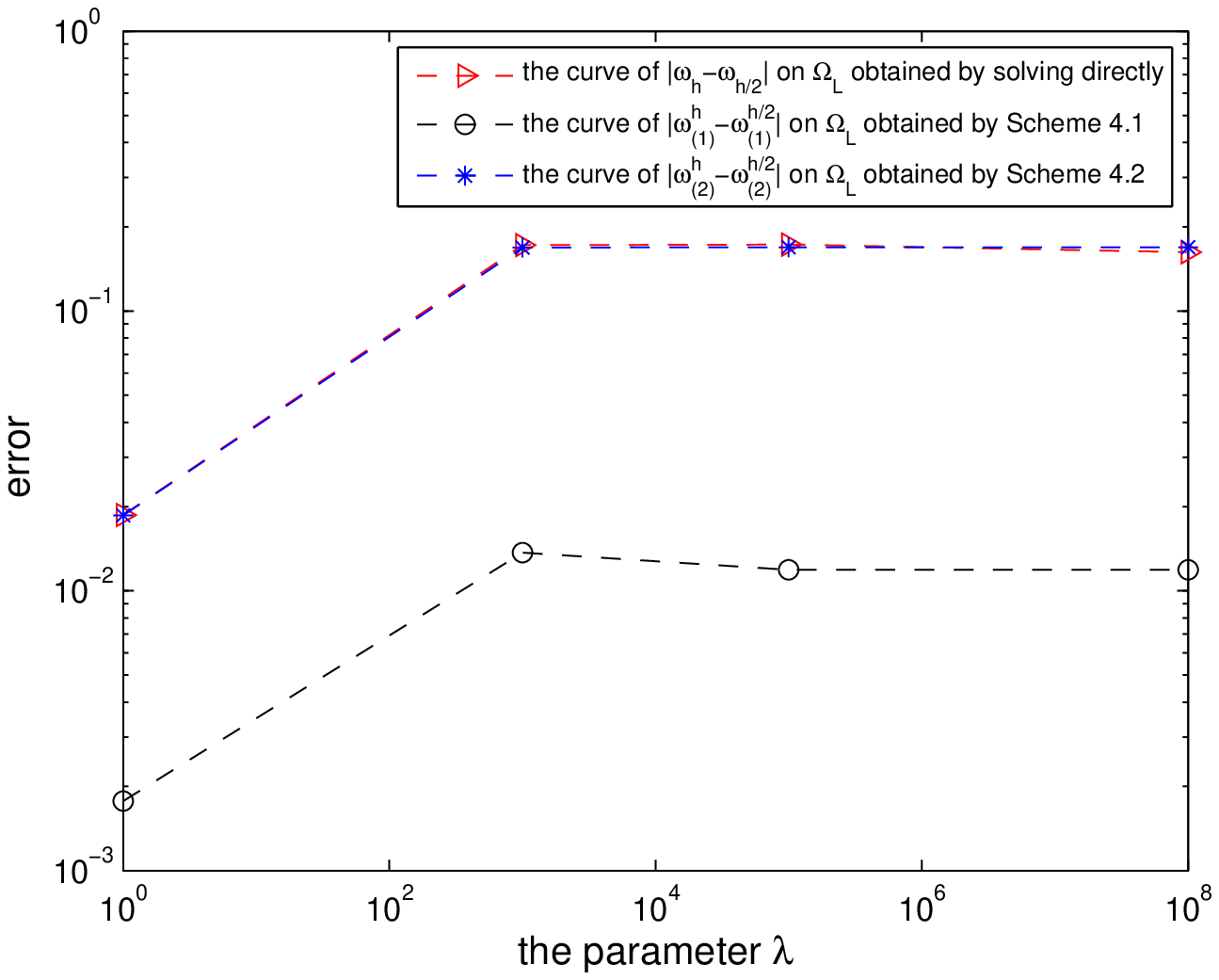}
  \caption{the error curves in the unit square (left) and the error curves in the L-shaped domain (right)}
\end{figure}

\bigskip
\noindent{\bf Acknowledgments.}
This work is supported by National Natural Science Foundation
of China (Grant no. 11761022).

\end{document}